\documentclass[final,1p,times]{elsarticle}

\usepackage{amssymb}

\usepackage{amsthm}
\usepackage{mathrsfs}

\usepackage{mathptmx}
\usepackage{amsmath,amsfonts}
\usepackage{url,hyperref,multirow}
\usepackage{float}
\usepackage{color}
\usepackage{epsfig,bm,caption2,epstopdf}
\usepackage{float}

\usepackage{diagbox}
\usepackage{setspace}
\usepackage{exscale}
\usepackage{relsize}
\usepackage{booktabs}
\usepackage{ifpdf}

\usepackage{dsfont}

\newtheorem{theorem}{Theorem}[section]
\newtheorem{remark}{Remark}[section]
\newtheorem{corollary}{Corollary}[section]
\newtheorem{lemma}{Lemma}[section]
\newtheorem{example}{Example}[section]

\newtheorem{assumption}{Assumption}

\journal{}

\begin{document}
	
	\begin{frontmatter}

\title{Local convergence analysis of a linearized Alikhanov scheme for the time fractional sine-Gordon equation}
\author[OUC]{Chang Hou}
\ead{1546758808@qq.com}

\author[OUC]{Hu Chen \corref{chen}}
\ead{chenhu@ouc.edu.cn} \cortext[chen]{Corresponding author.}

\address[OUC]{School of Mathematical Sciences, Ocean University of China, Qingdao 266100, China}
\begin{abstract}
  This paper investigates the time fractional sine-Gordon equation whose solution exhibits a weak singularity of type $t^{\alpha}$.
  By means of the Alikhanov formula we derive a fully discrete, linearized scheme.
  Using the more general regularity assumption, we derive a sharp truncation-error bound for the fractional derivative.
  Furthermore, we prove a key inequality and a less restrictive stability result that is valid on general graded temporal meshes.
  Consequently, the temporal local convergence order is shown to be $\min\{2,r\}$ in $H^1$-seminorm, where $r$ is the degree of grading; numerical experiments confirm that the optimal rate is already attained as soon as $r=2$.
\end{abstract}

\end{frontmatter}

 \section{Introduction}
 The sine-Gordon equation was first introduced by Bour in 1862. It has been extensively used to model various physical phenomena\cite{MR4850632}. Compared with the integer-order sine-Gordon equation, the diffusion-wave equations are capable of capturing physical processes whose response interpolates between that of diffusion and that of wave propagation\cite{MR4309151}. In this paper, we consider the time fractional sine-Gordon equation:
	\begin{subequations}
		\label{equation1}
	\begin{align}
		& D_{t}^{\alpha }u(\bm{x},t)-\nu\Delta u(\bm{x},t)+f(u(\bm{x},t))=0,\ (\bm{x},t)\in \Omega\times \left(0,T\right],\\
		&u(\bm{x},0)=\psi(\bm{x}),\ u_t(\bm{x},0)=\phi(\bm{x}),\ \bm{x}\in \Omega,\\
		&u(\bm{x},t)=0,\ (\bm{x},t)\in \partial\Omega\times \left[0,T\right],
	\end{align}
	\end{subequations}
where $\nu$ is a positive constant, $\Omega\subset \mathbb{R}^2$ is an
open bounded spatial domain $(0,L)^2$ with a boundary $\partial\Omega$. $f(u)$ is the function $\kappa^2\sin(u)$. $D_t^\alpha$ denotes the Caputo derivative of order $1<\alpha<2$, which is defined by
\begin{align}
	D_t^{\delta} v(\bm{x},t)=\frac{1}{\Gamma(2-\delta)}\int_{0}^{t}{(t-s)^{1-\delta}\frac{\partial^2 v(\bm{x},s)}{\partial s^2}ds} \, \ for\ 1<\delta<2.
\end{align}

For the diffusion-wave equations, it is often necessary to reduce the order of the fractional derivative term in analysis. Lyu et al. \cite{MR4065754} achieved order reduction by introducing an auxiliary variable $v=u_t$, and established a linearized scheme based on fast $L2$-$1_\sigma$ formula which is shown to have second-order unconditional convergence under discrete $H^1$-norm for nonlinear multi-term diffusion-wave equations.
Applying the same order-reduction method and Alikhanov formula, Du et al. \cite{MR4358469} developed a temporal second-order finite difference scheme; they used the discrete energy technique to prove the unconditional stability of the ADI finite difference scheme and  compact ADI finite difference scheme on uniform time meshes for the linear variable-order  time fractional  wave equations. Instead of discretizing the fractional derivative directly, Huang et al. \cite{MR4146516} exploited the backward-Euler formula based on the equivalent partial integro-differential equations and constructed two discrete schemes that require weaker smoothness of the true solution in time for nonlinear diffusion-wave equations.
However, both of the above works did not account for the solutions' weak singularity. In 2016, Jin et al. \cite{MR3449907} described the regularity of solutions to the linear diffusion-wave equation; under suitable conditions (cf.  Theorem A.2 in \cite{MR3449907}), the solution may satisfy $\|\partial_t^mu\|_{L_2}\le Ct^{\alpha-m}$ for $m\ge 1$.

To overcome the singularity of the solution at the initial time, nonuniform meshes are of greater scientific value. Lyu et al. \cite{MR4483525} argued that the fixed-order reduction technique encounters difficulties when applied to general (nonuniform) meshes, they proposed a novel symmetric order-reduction technique and applied it to develop two linearized discrete schemes for nonlinear diffusion-wave equations, where the L1 and Alikhanov methods are used for temporal discretization, respectively, establishing global error estimates in the sense of $H^2$-norm. Subsequently, this reduction approach has been widely adopted. An et al. \cite{MR4434375} proved the regularity of the exact solution of the linear diffusion-wave equations,  that is  $\|\frac{\partial^lu(\cdot,t)}{\partial t^l}\|_{H^2}\le C(1+t^{\alpha-l})$, and they applied the same order-reduction technique to the time fractional derivative and used a finite element method in space, ultimately deriving an $\alpha$-robust error analysis in the $H^1$-seminorm.
Kundaliya and Chaudhary \cite{MR4641576} develop a linearized, fully discrete method that couples finite element method in space with the L1 scheme on standard graded meshes for a time fractional nonlocal diffusion-wave equation of Kirchhoff type.
Zhang et al. \cite{MR4591717} employed a finite difference method for linear diffusion-wave equation, using the L1 scheme on graded meshes in time and a compact difference scheme in space, and they derived a local temporal error estimate of the fully discrete scheme. However, the local convergence analysis of nonlinear diffusion-wave equations has received little attention.

The time fractional sine-Gordon is a special case of nonlinear diffusion-wave equations.
Huang et al. \cite{MR4801462} applied symmetric order-reduction technique to the time fractional sine-Gordon equation and modified the standard L1 scheme  on standard graded temporal meshes, ultimately deriving pointwise error estimates by means of a stability result.
 It is worth noting that the numerical scheme in \cite{MR4801462} has not been linearized. Moreover, adopting the Alikhanov scheme for discretization may increase the convergence order in time.
For the time fractional sine-Gordon equation with weakly singular solutions, pointwise error estimates of the linearized  Alikhanov scheme on general time meshes are still lacking. In this work, we will fill this gap.

We first introduce a new auxiliary variable $p$, reducing the original equation to a coupled system. The exact solution and the auxiliary variable are assumed to satisfy Assumption \ref{assumption1}. In time, we adopt  Alikhanov scheme on general graded meshes, while in space we use the standard second order central difference.
 Under the more general regularity assumption, we estimate the truncation error of the discrete fractional derivative based on the Alikhanov scheme and we observe that the according coefficient is $\alpha$-robust. Building on our previous work \cite{submittied}, we derived a novel stability result subject to fewer constraints (see Lemma \ref{lemma6}). Then, by further proving an important Lemma \ref{keylemma}, we are able to use the stability result to derive local error estimates.

 The paper is organized as follows.
 Section \ref{sec:full} derives the fully discrete scheme by first rewriting the problem as a symmetric fractional-order system and then applying the Alikhanov formula to the coupled equations.
 In Section \ref{sec:trun} the truncation errors are estimated under general regularity assumptions.
   A more widely applicable stability result is obtained in Section \ref{sec:sta}. An important inequality is established yielding the final error estimates in Section \ref{sec:con}.
 The theoretical findings are corroborated in Section \ref{sec:num} by two numerical experiments that illustrate the sharpness of the local error estimates.

 \textbf{Notation.} Throughout this work, the symbol $C$ denotes a generic positive constant that is independent of the mesh sizes and takes different values in different places. $a\lesssim b$ signifies that there exists a positive constant C such that $a\le Cb$. $a\simeq b$ means that $a\lesssim b$ and $b\lesssim a$.
 \section{Fully discrete scheme}\label{sec:full}
 In this section, we construct a linearized discrete scheme: time is handled by Alikhanov scheme; space is discretized with the finite difference method.
\begin{lemma}(\cite{MR4483525},Lemma 2.1)
	\label{lemma1}
	For $\alpha\in(1,2)$ and $v(t)\in C^1([0,T])\cap C^2((0,T])$, it holds that
	\begin{align}
		D_t^{\alpha}v(t)=D_t^{\frac{\alpha}{2}}\left(D_t^{\frac{\alpha}{2}}v(t)\right)-v'(0)\frac{t^{1-\alpha}}{\Gamma(2-\alpha)}.
	\end{align}
\end{lemma}
Defining $\bar{u}(\bm{x},t)=u(\bm{x},t)-t\phi(\bm{x})$, equation \eqref{equation1} is equivalent to the following expression:
\begin{subequations}
	\label{equation2}
	\begin{align}
		& D_{t}^{\alpha }\bar u(\bm{x},t)-\nu\Delta \bar u(\bm{x},t)+\kappa^2\sin(\bar u(\bm{x},t)+t\phi(\bm{x}))=\nu t\Delta \phi(\bm{x}),\ (\bm{x},t)\in \Omega\times \left(0,T\right],\\
		&\bar u(\bm{x},0)=\psi(\bm{x}),\ \bar u_t(\bm{x},0)=0,\ \bm{x}\in \Omega,\\
		&\bar u(\bm{x},t)=0,\ (\bm{x},t)\in \partial\Omega\times \left[0,T\right].
	\end{align}
\end{subequations}
Let $\beta=\frac{\alpha}{2}$, and  $p(\bm{x},t)=D_t^\beta \bar{u}(\bm{x},t)$. Applying Lemma \ref{lemma1}, one has $D_t^\alpha \bar{u}(\bm{x},t)=D_t^\beta \left(D_t^\beta \bar{u}(\bm{x},t)\right)=D_t^\beta p(\bm{x},t)$. Furthermore, equation \eqref{equation1} can be rewritten as a coupled form:
\begin{subequations}
	\label{equation3}
	\begin{align}
		& D_{t}^{\beta} p(\bm{x},t)-\nu\Delta \bar u(\bm{x},t)+f(\bar u(\bm{x},t)+t\phi(\bm{x}))=\nu t\Delta \phi(\bm{x}),\ (\bm{x},t)\in \Omega\times \left(0,T\right],\\
		&D_{t}^{\beta}\bar u(\bm{x},t)=p(\bm{x},t),\ (\bm{x},t)\in \Omega\times \left(0,T\right],\\
		&\bar u(\bm{x},0)=\psi(\bm{x}),\ p(\bm{x},0)=0,\ \bm{x}\in \Omega,\\
		&\bar u(\bm{x},t)=0,\ p(\bm{x},t)=0,\ (\bm{x},t)\in \partial\Omega\times \left[0,T\right].
	\end{align}
\end{subequations}

\begin{assumption}
	\label{assumption1}
	The solutions $\bar{u}$ and $p$ of  \eqref{equation3} satisfy that $\bar u(\bm{x},t)\in C^{4,1}_{\bm{x},t}(\bar{\Omega}\times [0,T])$,   $p(\bm{x},t)\in C_{\bm{x},t}^{2,0}(\bar{\Omega}\times[0,T])$,
	\begin{align}
		\label{eq1}
		\sum_{0\le j\le 2}\left|\partial_{\bm{x}}^j\frac{\partial^l\bar{u}(\bm{x},t)}{\partial t^l}\right|\le C(1+t^{\alpha-l}),\ for\ l=0,1,2,3,\\
		\label{eq2}
		\sum_{0\le j\le 2}\left|\partial_{\bm{x}}^j\frac{\partial^lp(\bm{x},t)}{\partial t^l}\right|\le C(1+t^{\alpha/2-l}),\ for\ l=0,1,2,3,
	\end{align} where $\bm{x}\in\bar\Omega$ and $t\in (0,T]$. The notation $\partial_{\bm{x}}^j$  is defined by  $\partial_{\bm{x}}^j v=\frac{\partial^{j_1+j_2}v}{\partial x^{j_1}\partial y^{j_2}}$ where $j_1$ and $j_2$ are non-negative constants satisfying $j_1+j_2=j$.
\end{assumption}

Setting $\tau_n=t_n-t_{n-1}$, we consider the quasi-graded temporal meshes $\{t_n\}_{n=0}^N$, i.e.
\begin{align}
	\label{mesh}
	\tau_1\simeq N^{-r},\ t_n\simeq\tau_1n^r,\ \tau_n\simeq\tau_1^{1/r}t_n^{1-1/r},\ for\  n=1,2,\dots, N,
\end{align}
  where $r\ge 1$ represents the degree of grading. Denote $\max_{1\le n\le N}\{\tau_n\}$  by $\tau$. The rectangular spatial domain is uniformly discretized into the mesh
  $\bar\Omega_h=\bigl\{(x_i,y_j)\mid 0\le i,j\le M\bigr\}$,
  with the spatial size $h=L/M$ in the $x$ and $y$ directions.
 Let $\Omega_h=\bar\Omega_h\cap\Omega$, and  $\partial\Omega_h=\bar\Omega_h\cap\partial\Omega$.

In the temporal direction, the discretization is performed at the off-set time points $t_n^*=t_n-(1-\sigma) \tau_{n}$ for $n=1,2,\dots, N$, where $0\le\sigma\le 1$.  Applying linear interpolation, we obtain the approximation $v(t_n^*)\approx v^{n,*}:=\sigma v^n+(1-\sigma)v^{n-1}$. Discretizing equation \eqref{equation3} at time $t_n^*$, yields
\begin{subequations}
	\label{equation4}
	\begin{align}
		& \delta_{t}^{\beta,*} p_{i,j}^{n}-\nu\Delta_h \bar u_{i,j}^{n,*}+F(u_{i,j}^{n,*})=\nu t_n^*\Delta \phi_{i,j}+(R_1)_{i,j}^n,\ for\ (x_i,y_j)\in \Omega_h,1\le n\le N,\\
		&\delta_{t}^{\beta,*}\bar u_{i,j}^n=p_{i,j}^{n,*}+(R_2)_{i,j}^n,\ u_{i,j}^n=\bar{u}_{i,j}^{n}+t_n \phi_{i,j},\ for\ (x_i,y_j)\in \Omega_h,1\le n\le N,\\
		&\bar u_{i,j}^0=\psi_{i,j},\ p_{i,j}^0=0,\ for\ (x_i,y_j)\in \Omega_h,\\
		&\bar u_{i,j}^n=0,\ p_{i,j}^n=0,\ for\ (x_i,y_j)\in\partial\Omega_h, 0\le n\le N,
	\end{align}
\end{subequations}
where $F(u_{i,j}^{n,*}):=f(u_{i,j}^{n-1})+\sigma f'(u_{i,j}^{n-1})(u_{i,j}^n-u_{i,j}^{n-1})$ and $\Delta_h$ is the standard ﬁve-point ﬁnite difference operator to approximate the Laplacian operator $\Delta$.
$\delta_{t}^{\beta,*} v^{n}=\sum_{k=1}^{n}g_{n,k}(v^{k}-v^{k-1})$ is the approximation of the Caputo derivative $D_t^{\beta}v(t_n^*)$ with $\frac{1}{2}<\beta<1$, that is
\begin{align*}
	\delta_{t}^{\beta,*} v^{n}&=\frac{1}{\Gamma(1-\beta)}\sum_{k=1}^{n-1}\int_{t_{k-1}}^{t_{k}}{(t_n^*-s)^{-\beta}\partial_s\Pi_{2,k} v(s)ds}+\frac{1}{\Gamma(1-\beta)}\int_{t_{n-1}}^{t_n^*}{(t_n^*-s)^{-\beta}\partial_s\Pi_{1,n}v(s)ds},
\end{align*}
where $\Pi_{1,k}$ and $\Pi_{2,k}$ denote the linear and quadratic interpolation operators, respectively.
For the Alikhanov scheme to hold, $\sigma$ is required to be $1-\frac{\beta}{2}$. Define the coefficients
 \begin{align*}
 	a_{n,n}&=\frac{\sigma^{1-\beta}}{\Gamma(2-\beta)}\tau_{n+1}^{1-\beta},\ for\ n\ge0,\\
 	a_{n,k}&=\frac{1}{\Gamma(1-\beta)}\int_{t_k}^{t_{k+1}}{(t_{n+1}^*-\eta)^{-\beta}d\eta},\ for\ n\ge 1\ and\ 0\le k\le n-1,\\
 	b_{n,k}&=\frac{1}{\Gamma(1-\beta)}\frac{2}{t_{k+2}-t_k}\int_{t_k}^{t_{k+1}}{(t_{n+1}^*-\eta)^{-\beta}(\eta-t_{k+1/2})d\eta},\ for\ n\ge 1\ and\ 0\le k\le n-1.
 \end{align*}  When $n=1$, one has $g_{1,1}=\tau_1^{-1}a_{0,0}$. For $n\ge 2$, the section 2 in \cite{MR3936261} gives that
\begin{align*}
	g_{n,k}=\begin{cases}
		\tau_k^{-1}(a_{n-1,0}-b_{n-1,0}),\ if\ k=1,\\
		\tau_k^{-1}(a_{n-1,k-1}+b_{n-1,k-2}-b_{n-1,k-1}),\ if\ 2\le k\le n-1,\\
		\tau_k^{-1}(a_{n-1,n-1}+b_{n-1,n-2}),\ if\ k=n.
	\end{cases}
\end{align*}

Neglecting the truncation error terms yield the fully discrete scheme:
\begin{subequations}
	\label{equation5}
	\begin{align}
		& \delta_{t}^{\beta,*} P_{i,j}^{n}-\nu\Delta_h \bar U_{i,j}^{n,*}+F(U_{i,j}^{n,*})=\nu t_n^*\Delta \phi_{i,j},\ for\ (x_i,y_j)\in \Omega_h,1\le n\le N,\\
		&\delta_{t}^{\beta,*}\bar U_{i,j}^n=P_{i,j}^{n,*},\ U_{i,j}^n=\bar{U}_{i,j}^{n}+t_n \phi_{i,j},\ for\ (x_i,y_j)\in \Omega_h,1\le n\le N,\\
		&\bar U_{i,j}^0=\psi_{i,j},\ P_{i,j}^0=0,\ for\ (x_i,y_j)\in \Omega_h,\\
		&\bar U_{i,j}^n=0,\ P_{i,j}^n=0,\ for\ (x_i,y_j)\in\partial\Omega_h, 0\le n\le N.
	\end{align}
\end{subequations}
 Assume that the coefficients $g_{n,k}$ satisfy the following properties:\\
P1: $g_{n,k+1}>g_{n,k}>0$, for $1\le k\le n-1\le N-1$;\\
P2: $(2\sigma-1)g_{n,n}> \sigma g_{n,n-1}$, for $2\le n\le N$;\\
P3: $g_{n,k}\ge \frac{1}{m_c\tau_k}\int_{t_{k-1}}^{t_k}\frac{(t_n-s)^{-\beta}}{\Gamma(1-\beta)}ds$ with $m_c>0$, for $1\le k\le n\le N$;\\
P4: There exists a constant $\rho$ such that  $\max_{1\le k\le N-1}{\tau_k/\tau_{k+1}}\le \rho$.

\section{Truncation errors}\label{sec:trun}
In this section we carry out the analysis of truncation errors under the more general regularity assumption. Because of the mesh conditions \eqref{mesh},  we have $t_{k-1}\simeq t_{k}$ and $\tau_{k-1}\simeq\tau_{k}$ for $k\ge 2$. There exist positive constants $C_m$, $C_l$, $C_r$, $C_a$ and $C_b$ such that the following holds:
\begin{align}
t_{k}\le t_{k+1}\le C_m t_{k},\ C_{l}\tau_{k}\le \tau_{k+1}\le C_r\tau_{k},\  C_{a}\tau_1^{1/r}t_k^{1-1/r}\le\tau_k\le C_{b} \tau_1^{1/r}t_k^{1-1/r},\ for\ k\ge 1.
\end{align}
\begin{lemma}
	\label{r1}
	Assume that $\sigma=1-\frac{\beta'}{2}$ and $\beta'\in(0,1)$. Let $v$ satisfy $\left|\frac{\partial^lv(t)}{\partial t^l}\right|\le C(1+t^{\mu-l})$ with $\mu\in(0,1)\cup(1,2)$ for $t>0$, when $l=0,1,2,3$. For $n=1,2,\dots,N$, one has
	\begin{align}
		\label{equation26}
		\left|D_t^{\beta'}v(t_n^*)-\delta_t^{\beta',*}v^n\right|\le C\tau_1^{\mu-\beta'}(\tau_1/t_n)^{\min\{1+\beta',\frac{3-\beta'}{r}-\mu+\beta'\}}.
	\end{align}
	\begin{proof}
	An estimate for the truncation error incurred by the L1 scheme in discretizing the fractional-derivative term has been estimated in Lemma 3.1 of \cite{2025arXiv250612954K}.  We aim to analyze the error associated with the Alikhanov scheme.
		Define
		\begin{align*}
		\varpi(s):=\begin{cases}
				\Pi_{2,k}v(s)-v(s),\ when\ s\in (t_{k-1},t_k)\ and\ 1\le k\le n-1,\\
				\Pi_{2,n^*}v(s)-v(s)\ when\ s\in (t_{n-1},t_n^*).
			\end{cases}
		\end{align*}  Firstly, the estimate of $\varpi(s)$ is considered for $s\in (t_{k-1},t_k),\ 2\le k\le n-1$.
	\begin{equation}
			\begin{aligned}
			\label{equation14}
			\left|\varpi(s)\right|&=\left|\Pi_{2,k}v(s)-v(s)\right|\\
			&=\left|\frac{v^{(3)}(\zeta_1)}{6}(s-t_{k-1})(s-t_k)(s-t_{k+1})\right|,\ \zeta_1\in[t_{k-1},t_{k+1}]\\
			&\le C(C_r+1)\frac{T^{3-\mu}+1}{C_m^{\mu-3}}t_{k}^{\mu-3}\tau_{k}^3:= Z_1 t_{k}^{\mu-3}\tau_{k}^3\min\{1,\frac{(t_n^*-s)}{\sigma\tau_n}\}.
		\end{aligned}
	\end{equation}
The inequality follows from the conditions $\tau_{k}\simeq \tau_{k+1}$, and $t_{k-1}\simeq t_k$.	
For $s\in (t_{n-1},t_n^*), n\ge2$, using $t_{n-1}\simeq t_n$, it holds that
\begin{equation}
	\begin{aligned}
		\label{equation15}
		\left|\varpi(s)\right|&=\left|\Pi_{2,n^*}v(s)-v(s)\right|\\
		&=\left|\frac{v^{(3)}(\zeta_2)}{6}(s-t_{n-1})(s-t_n^*)(s-t_{n})\right|,\ \zeta_2\in[t_{n-1},t_{n}]\\
		&\le C\frac{T^{3-\mu}+1}{C_m^{\mu-3}} t_{n}^{\mu-3}\tau_{n}^2(t_n^*-s)\\
		&\le C \frac{1}{\sigma}\frac{T^{3-\mu}+1}{C_m^{\mu-3}}t_n^{\mu-3}\tau_{n}^2(t_n^*-s):=Z_2t_{n}^{\mu-3}\tau_{n}^3\min\{1,\frac{(t_n^*-s)}{\sigma\tau_n}\}.
	\end{aligned}
\end{equation}
We set $t_1^{**}=\min\{t_1,t_n^*\}$.  Owing to $\varpi (t_1^{**})=0$, it is easy to get $\varpi(s)=-\int_{s}^{t_1^{**}}\varpi'(t)dt$, for $s\in (0,t_1^{**})$. When $\mu\in(0,1)$, we get
\begin{equation}
	\begin{aligned}
		\label{eqaution10}
		\left|\varpi'(t)\right|&\le \left|\partial_t(\Pi_{2,1^{**}} v(t))\right|+\left|\partial_tv(t)\right|\\
		&\le \left(2\sigma^{-1}+\max\{C_l^{-1},(1-\sigma)^{-1}\}\right)(t_1)^{-1}osc(v,\left[0,t_2\right])+C(1+t^{\mu-1})\\
		&\le \left(2\sigma^{-1}+\max\{C_l^{-1},(1-\sigma)^{-1}\}\right)(t_1)^{-1}osc(v,\left[0,t_2\right])+C(T^{1-\mu}+1)s^{\mu-1},
	\end{aligned}
\end{equation}
where $osc(v,[0,t_2])$ denotes the oscillation of the function $v(t)$ over the interval $[0,t_2]$. Employing $osc(v,[0,t_2])\le \int_{0}^{t_2}\left|v'(t)\right|dt\le C(T^{1-\mu}+1)\int_{0}^{t_2}t^{\mu-1}dt=C\frac{(T^{1-\mu}+1)}{\mu} t_2^{\mu}\le C\frac{(T^{1-\mu}+1)C_m^{\mu}}{\mu} t_1^{\mu}$, we obtain $\left|\varpi(s)\right|\le  C(t_1^{**}-s)(t_1^{\mu-1}+s^{\mu-1})$, where constant $C$ is bounded when $\beta'\rightarrow 1$ or $\mu\rightarrow 1$.
In the alternative case where $\mu\in(1,2)$, \cite{2025arXiv250612954K} introduced a crucial technique, namely, rewriting $\varpi(s)$ as $\Pi_{2,1^{**}}\bar v(s)-\bar v(s)$, where $\bar v(s)=v(s)-v(0)-v'(0)s$. With the aid of $\left|\bar v'(t)\right|=\left|v'(t)-v'(0)\right|\le \int_{0}^{t}\left|v''(s)\right|ds \le C(T^{2-\mu}+1)\int_{0}^{t}s^{\mu-2}ds\le C\frac{(T^{2-\mu}+1)}{\mu-1} t^{\mu-1}$, by analogy with inequality \eqref{eqaution10} , we estimate $|\varpi'(t)|$ once more to obtain:  $\left|\varpi(s)\right|\le C (t_1^{**}-s)t_1^{\mu-1}$ where constant $C$ is bounded when $\beta'\rightarrow 1$ or $\mu\rightarrow 2$. In a conclusion, it holds that $\left|\varpi(s)\right|\le Z_3 (t_1^{**}-s)(t_1^{\mu-1}+s^{\mu-1})$ for $\mu\in(0,1)\cup(1,2)$.

For convenience,	$D_t^{\beta'}v(t_n^*)-\delta_t^{\beta',*}v^n$ is denoted by $r_1^n$. In view of the definitions of $D_t^{\beta'}v(t_n^*)$ and $\delta_t^{\beta',*}v^n$, and equation (4.2) from Reference \cite{MR4085134} , integration by parts yields
\begin{equation}
	\begin{aligned}
		\label{equation11}
		\Gamma(1-\beta')r_1^n&=\sum_{k=1}^{n-1}\int_{t_{k-1}}^{t_k}(t_n^*-s)^{-\beta'}\partial_s\left(v(s)-\Pi_{2,k}v(s)\right)ds+\int_{t_{n-1}}^{t_n^*}(t_n^*-s)^{-\beta'}\partial_s\left(v(s)-\Pi_{2,n^*}v(s)\right)ds\\
		&=-\beta'\sum_{k=1}^{n-1}\int_{t_{k-1}}^{t_k}(t_n^*-s)^{-\beta'-1}\left(v(s)-\Pi_{2,k}v(s)\right)ds-\beta'\int_{t_{n-1}}^{t_n^*}(t_n^*-s)^{-\beta'-1}\left(v(s)-\Pi_{2,n^*}v(s)\right)ds \\
&\qquad -\lim_{\epsilon\rightarrow0}{\epsilon^{-\beta'}\varpi(t_n^*-\epsilon)}\\
		&=\beta'\int_{0}^{t_n^*}(t_n^*-s)^{-\beta'-1}\varpi(s)ds.
	\end{aligned}
\end{equation}
We now proceed by splitting the analysis into two cases, estimating the right-hand side of the equation \eqref{equation11}.
Case $n\ge 2$: Using $\frac{\tau_1}{t_n^*}\ge \frac{\tau_1-s}{t_n^*-s}$ for $0<s<t_1$, we have
\begin{equation}
	\begin{aligned}
		\label{equation12}
		\int_{0}^{t_1}(t_n^*-s)^{-\beta'-1}|\varpi(s)|ds&\le  Z_3 \int_{0}^{t_1}(t_n^*-s)^{-\beta'-1}(t_1-s)(t_1^{\mu-1}+s^{\mu-1})ds\\
		&\le Z_3(\tau_1/t_n^*)^{\beta'+1} \int_{0}^{t_1}(t_1-s)^{-\beta'}(t_1^{\mu-1}+s^{\mu-1})ds\\
		&=Z_3\tau_1^{\mu-\beta'}(\tau_1/t_n^*)^{\beta'+1} \int_{0}^{1}(1-\hat s)^{-\beta'}(1+\hat s^{\mu-1})d\hat s\\
		&=Z_3 \tau_1^{\mu-\beta'}(\tau_1/t_n^*)^{\beta'+1}\left(\frac{1}{1-\beta'}+\frac{\Gamma(\mu)\Gamma(1-\beta')}{\Gamma(\mu+1-\beta')}\right)\\
		&\le Z_3 \sigma^{-\beta'-1} \tau_1^{\mu-\beta'}(\tau_1/t_n)^{\beta'+1}\left(\frac{1}{1-\beta'}+\frac{\Gamma(\mu)\Gamma(1-\beta')}{\Gamma(\mu+1-\beta')}\right).
	\end{aligned}
\end{equation}
Let $\theta:=\frac{1}{2}(\tau_1/t_n^*)^{1/r}\le \frac{1}{2}$.  Applying $C_{a}  \tau_1^{1/r}t_k^{1-1/r}\le \tau_k\le C_{b}  \tau_1^{1/r}t_k^{1-1/r}$ for $k\ge 1$ and
$	t_k^{-1-\beta'/r}\le  s^{-1-\beta'/r},\ for\  s\in(t_{k-1},t_{k}),\ k\ge2$,
  one has
\begin{equation}
	\begin{aligned}
		\label{equation13}
		\int_{t_1}^{t_n^*}(t_n^*-s)^{-\beta'-1}|\varpi(s)|ds
		&\le \max\{Z_1,Z_2\}\sigma^{-1}C_{b}^3\tau_1^{3/r}\max_{n\ge k\ge 2}\{t_k^{1+\mu+\frac{\beta'-3}{r}}\} \int_{t_1}^{t_n^*}(t_n^*-s)^{-\beta'-1}s^{-1-\beta'/r}\min\{1,\frac{(t_n^*-s)}{\tau_n}\}ds\\
		&\le  Z_4\tau_1^{3/r}\max_{n\ge k\ge 2}\{t_k^{1+\mu+\frac{\beta'-3}{r}}\} \int_{t_1}^{t_n^*}(t_n^*-s)^{-\beta'-1}s^{-1-\beta'/r}\min\{(1-\theta)^{1/r-1},\frac{2(t_n^*-s)}{\tau_1^{1/r}s^{1-1/r}}\}ds\\
		&\le Z_4\max_{n\ge k\ge 2}\{t_k^{1+\mu+\frac{\beta'-3}{r}}\}(I_1+I_2),
    \end{aligned}
\end{equation}
where $Z_4=\max\{Z_1,Z_2\}\sigma^{-1}C_{b}^3\max\{1,C_a^{-1}\}$.
We split the interval $(t_1,t_n^*)$ into two sub-intervals and estimate each part separately. If $t_n^*(1-\theta)\le t_1$,  this term $I_1$ vanishes.
\begin{equation}
	 \begin{aligned}
		\label{equation16}
		I_1&=(1-\theta)^{1/r-1}\tau_1^{3/r}\int_{t_1}^{t_n^*(1-\theta)}(t_n^*-s)^{-\beta'-1}s^{-1-\beta'/r}ds\\
		&\le 2\tau_1^{3/r}(t_n^*)^{-\beta'/r-\beta'-1}\int_{t_1/t_n^*}^{(1-\theta)}(1-\hat s)^{-\beta'-1}\hat s^{-1-\beta'/r}d\hat s\\		
		&\le 2\tau_1^{3/r}(t_n^*)^{-\beta'/r-\beta'-1}\left(\frac{2^{1+\beta'}r}{\beta'}(\tau_1/t_n^*)^{-\beta'/r}+\frac{2^{1+\beta'/r}}{\beta'}\theta^{-\beta'}\right)\\
		&= 2\tau_1^{3/r-\beta'/r}(t_n^*)^{-\beta'-1}\left(\frac{2^{1+\beta'}r}{\beta'}+\frac{2^{1+\beta'/r+\beta'}}{\beta'}\right).
	\end{aligned}
\end{equation}
The estimate for $I_2$ is as follows:
\begin{equation}
	 \begin{aligned}
		\label{equation17}
		I_2&=2\tau_1^{2/r}\int_{t_n^*(1-\theta)}^{t_n^*}(t_n^*-s)^{-\beta'-1}s^{-2+\frac{1-\beta'}{r}}(t_n^*-s)ds\\
		&= 2\tau_1^{2/r}(t_n^*)^{-1-\beta'+\frac{1-\beta'}{r}}\int_{(1-\theta)}^{1}(1-\hat s)^{-\beta'}\hat s^{-2+\frac{1-\beta'}{r}}d\hat s\\
		&\le \frac{2^{3-\frac{1-\beta'}{r}}}{1-\beta'} \tau_1^{2/r}(t_n^*)^{-1-\beta'+\frac{1-\beta'}{r}}\theta^{1-\beta'}\\
		&=\frac{2^{2+\beta'-\frac{1-\beta'}{r}}}{1-\beta'} \tau_1^{\frac{3-\beta'}{r}}(t_n^*)^{-1-\beta'}.
	\end{aligned}
\end{equation}
 Furthermore, we get
\begin{equation}
	 \begin{aligned}
		\int_{t_1}^{t_n^*}(t_n^*-s)^{-\beta'-1}|\varpi(s)|ds&\le Z_42^5r\sigma^{-\beta'-1}\max_{n\ge k\ge 2}\{t_k^{1+\mu+\frac{\beta'-3}{r}}\}\tau_1^{\frac{3-\beta'}{r}}\left(\frac{1}{\beta'}(t_n)^{-\beta'-1}+\frac{1}{1-\beta'} (t_n)^{-1-\beta'}\right)\\
		&= Z_5\sigma^{-\beta'-1}\max_{n\ge k\ge 2}\left\{t_n^{1+\beta'}\left(\frac{t_k}{t_n}\right)^{1+\beta'}t_k^{\frac{\beta'-3}{r}+\mu-\beta'}\right\}\tau_1^{\frac{3-\beta'}{r}}(t_n)^{-\beta'-1}\left(\frac{1}{\beta'}+\frac{1}{1-\beta'} \right)\\
		&=Z_5\sigma^{-\beta'-1}\tau_1^{\mu-\beta'}\max_{n\ge k\ge 2}\left\{\left(t_k/t_n\right)^{1+\beta'}(\tau_1/t_k)^{\frac{3-\beta'}{r}-\mu+\beta'}\right\}\left(\frac{1}{\beta'}+\frac{1}{1-\beta'} \right)\\
		&\le Z_5\sigma^{-\beta'-1}\tau_1^{\mu-\beta'}(\tau_1/t_n)^{\min\{1+\beta',\frac{3-\beta'}{r}-\mu+\beta'\}}\left(\frac{1}{\beta'}+\frac{1}{1-\beta'} \right),
	\end{aligned}
\end{equation}
where $Z_5=Z_42^5r$. Therefore, combining the results from \eqref{equation12} and \eqref{equation13}, for $n\ge 2$, we obtain
\begin{align*}
	\Gamma(1-\beta')|r_1^n|&\le Z_3 \sigma^{-\beta'-1} \tau_1^{\mu-\beta'}(\tau_1/t_n)^{\beta'+1}\left(\frac{\beta'}{1-\beta'}+\frac{\beta'\Gamma(\mu)\Gamma(1-\beta')}{\Gamma(\mu+1-\beta')}\right)\\
	&+Z_5\sigma^{-\beta'-1}\tau_1^{\mu-\beta'}(\tau_1/t_n)^{\min\{1+\beta',\frac{3-\beta'}{r}-\mu+\beta'\}}\left(1+\frac{\beta'}{1-\beta'} \right)\\
	&\le (Z_3+Z_5)\sigma^{-\beta'-1}\tau_1^{\mu-\beta'}(\tau_1/t_n)^{\min\{1+\beta',\frac{3-\beta'}{r}-\mu+\beta'\}}\left(1+\frac{\beta'}{1-\beta'} +\frac{\beta'\Gamma(\mu)\Gamma(1-\beta')}{\Gamma(\mu+1-\beta')}\right).
\end{align*}

Case $n=1$: We get
\begin{equation}
	\begin{aligned}
		\beta'\int_{0}^{t_1^*}(t_1^*-s)^{-\beta'-1}|\varpi(s)|ds&\le Z_3\beta' \int_{0}^{t_1^*}(t_1^*-s)^{-\beta'}(t_1^{\mu-1}+s^{\mu-1})ds\\
		&=Z_3
		\beta'\frac{1}{1-\beta'}(t_1^*)^{1-\beta'}t_1^{\mu-1}+Z_3\beta'(t_1^*)^{\mu-\beta'}\int_{0}^{1}(1-\hat s)^{-\beta'}\hat s^{\mu-1}d\hat s\\
		&=Z_3\left(\frac{\beta'}{1-\beta'}\sigma^{1-\beta'} +\frac{\beta'\Gamma(1-\beta')\Gamma(\mu)}{\Gamma(1-\beta'+\mu)}\sigma^{\mu-\beta'}\right)\tau_1^{\mu-\beta'}.
	\end{aligned}
\end{equation}

Combining these two cases, we obtain that for  \( n\ge 1 \),
\begin{align*}
	|r_1^n|\le 4(Z_3+Z_5) \left(\frac{1}{\Gamma(1-\beta')}+\frac{\beta'}{\Gamma(2-\beta')} +\frac{\beta'\Gamma(\mu)}{\Gamma(\mu+1-\beta')}\right)\tau_1^{\mu-\beta'}(\tau_1/t_n)^{\min\{1+\beta',\frac{3-\beta'}{r}-\mu+\beta'\}}.
\end{align*}
The proof is completed.
	\end{proof}
	\begin{remark}
		This proof shows that, when $\mu=\alpha$ or $\mu=\frac{\alpha}{2}$, the constant $C$ in \eqref{equation26} remains bounded as $\alpha\rightarrow2$ and is  therefore $\alpha$-robust.
	\end{remark}
\end{lemma}
\begin{lemma}
	\label{r3}
	 If $\sigma=1-\frac{\beta}{2}$, $f(\cdot)\in C^2(\mathbb{R})$ and $\left|\frac{\partial^lv(t)}{\partial t^l}\right|\le C(1+t^{\mu-l})$ with $\mu\in(0,1)\cup(1,2)$ for $t>0$, when $l=0,1,2,3$. For $1\le n\le N$, it holds that
	 \begin{align}
	 	\label{r2}
	 	|v^{n,*}-v(t_n^*)|&\le C(\tau_1/t_n)^{2/r}t_n^{\mu},\\
	 	\left|F(v^{n,*})-f(v(t_n^*))\right|&\le C(\tau_1/t_n)^{2/r}t_n^{\mu}.
	 \end{align}
	 \begin{proof}
	 According to Taylor expansion with integral reminder, one has	 	
	 	\begin{align*}
	 		&v(t_n)=v(t_n^*)+v'(t_n^*)(1-\sigma)\tau_n+\int_{t_n^*}^{t_n}{(t_n-s)v''(s)}ds,\\
	 		&v(t_{n-1})=v(t_{n}^*)-v'(t_{n}^*)\sigma\tau_n+\int_{t_{n}^*}^{t_{n-1}}{(t_{n-1}-s)v''(s)}ds.
	 	\end{align*}
	 	Thus, taking a linear combination, we obtain
	 	\begin{align}
	 		\label{eq22}
	 		v^{n,*}-v(t_n^*)=\sigma \int_{t_n^*}^{t_n}{(t_n-s)v''(s)}ds+(1-\sigma)\int_{t_{n}^*}^{t_{n-1}}{(t_{n-1}-s)v''(s)}ds.
	 	\end{align}
For the first integral in \eqref{eq22}, the following inequality holds for all $n\ge 1$, 	
	 	\begin{align*}
	 	\sigma	\left|\int_{t_n^*}^{t_n}{(t_n-s)v''(s)}ds\right|&\le \sigma\int_{t_n^*}^{t_n}{(t_n-s)|v''(s)|}ds\\
	 		&\le C\sigma\int_{t_n^*}^{t_n}{(t_n-s)(1+s^{\mu-2})}ds\\
	 	&	\le C\sigma(1+T^{2-\mu})\int_{t_n^*}^{t_n}{(t_n-s)s^{\mu-2}}ds\\
	 	&	\le C(1+T^{2-\mu})\sigma^{\mu-1}\tau_n^2 t_n^{\mu-2}\\
	 		&\le  C(1+T^{2-\mu})\sigma^{\mu-1}C_b^2(\tau_1/t_n)^{2/r} t_n^{\mu}:=A_1(\tau_1/t_n)^{2/r}t_n^{\mu}.
	 	\end{align*}	 	
	 	For the second integral in \eqref{eq22},  we have
	 	\begin{align*}
	 	(1-\sigma)	\left|\int_{t_{n}^*}^{t_{n-1}}{(t_{n-1}-s)v''(s)}ds\right|&\le (1-\sigma)	\int_{t_{n-1}}^{t_{n}^*}{(s-t_{n-1})|v''(s)|}ds\\
	 		&\le C(1-\sigma)(1+T^{2-\mu})\int_{t_{n-1}}^{t_{n}^*}{(s-t_{n-1})s^{\mu-2}}ds\\
	 		&\le C(1-\sigma)(1+T^{2-\mu})\tau_n^2t_{n-1}^{\mu-2}\\
	 		&\le C(1-\sigma)(1+T^{2-\mu})C_m^{2-\mu}\tau_n^2t_{n}^{\mu-2}\\
	 		&\le  C(1-\sigma)(1+T^{2-\mu})C_m^{2-\mu}C_b^2(\tau_1/t_n)^{2/r}t_{n}^{\mu},
	 	\end{align*}
	 	if $n\ge 2$, and
	 	\begin{align*}
	 			(1-\sigma)	\left|\int_{0}^{t_{1}^*}{sv''(s)}ds\right|&\le 	(1-\sigma)	\int_{0}^{t_{1}^*}{s|v''(s)|}ds\\
	 			&\le 	C(1-\sigma)(1+T^{2-\mu})	\int_{0}^{t_{1}^*}{ss^{\mu-2}}ds\\
	 			&=C(1-\sigma)(1+T^{2-\mu})\frac{(t_1^*)^{\mu}}{\mu}\\
	 			&=C(1-\sigma)(1+T^{2-\mu})\frac{\sigma^\mu}{\mu}(\tau_1/t_1)^{2/r}t_1^\mu,
	 	\end{align*}
	 	if $n=1$. Therefore, there exists a positive constant $A_2$ such that $(1-\sigma)	\left|\int_{t_{n}^*}^{t_{n-1}}{(t_{n-1}-s)v''(s)}ds\right|\le A_2(\tau_1/t_n)^{2/r}t_{n}^{\mu}$ for $n\ge 1$.
	 	
	 	In summary, the proof of inequality \eqref{r2} has been completed.
	 	
 The fact that $f(\cdot)\in C^2(\mathbb{R})$ implies
	\begin{align*}
	\left|F(v^{n,*})-f(v(t_n^*))\right|	&\le \left|f(v(t_n^*))-f(v^{n,*})\right|+\left|f(v^{n,*})-f(v^{n-1})-\sigma f'(v^{n-1})(v^n-v^{n-1})\right|\\
		&\le \left|f'(\xi_1)\right|\left|v(t_n^*)-v^{n,*}\right|+\left|\frac{1}{2}f''(\xi_2)\sigma^2(v^n-v^{n-1})^2\right|\\
		&\le (A_1+A_2)\left|f'(\xi_1)\right|(\tau_1/t_n)^{2/r}t_n^{\mu}+\left|\frac{1}{2}f''(\xi_2)\sigma^2\right|\left|v^n-v^{n-1}\right|^2\\
		&\le A_3\left[(\tau_1/t_n)^{2/r}t_n^{\mu}+\left|v^n-v^{n-1}\right|^2\right],
	\end{align*}
		where $\xi_1$ is between $v(t_n^*)$ and $v^{n,*}$, $\xi_2$ is between $v^{n-1}$ and $v^{n,*}$. Here, the constant $A_3$  is independent of mesh sizes and variable $n$. Next, we provide an estimate for $|v^n-v^{n-1}|$.
	\begin{align*}
		|v^n-v^{n-1}|&=\begin{cases}
			|v'(\xi_3)|\tau_n\le C\tau_n(1+\xi_3^{\mu-1})\le C\max\{1,C_m^{1-\mu}\}\tau_n(1+t_{n}^{\mu-1}),\ for\ n\ge2,\\
			|\int_{t_{0}}^{t_1}{v'(s)ds}|\le C\int_{t_{0}}^{t_1}{1+s^{\mu-1}ds}\le C(t_1+\frac{1}{\mu}t_1^{\mu}),\ for\ n=1,
		\end{cases}\\
		&\le  \left[2C\max\{1,C_m^{1-\mu}\}+C(1+\frac{1}{\mu})\right]\max\{1,T^{\mu-1},T^{1-\mu}\}(\tau_n/t_n)t_{n}^{\min\{\mu,1\}}\le A_4(\tau_1/t_n)^{1/r}t_{n}^{\min\{\mu,1\}},
	\end{align*}
	where $\xi_3$ is between $t_{n-1}$ and $t_{n}$. $A_4$ is $\left[2C\max\{1,C_m^{1-\mu}\}+C(1+\frac{1}{\mu})\right]\max\{1,T^{\mu-1},T^{1-\mu}\}C_b$. Therefore one has
	\begin{align*}
	\left|F(v^{n,*})-f(v(t_n^*))\right|	&\le A_3\left[(\tau_1/t_n)^{2/r}t_n^{\mu}+A_4^2(\tau_1/t_n)^{2/r}t_n^{\min\{2,2\mu\}}\right]\\
	&\le A_3 \left(1+A_4^2T^{\min\{2-\mu,\mu\}}\right)(\tau_1/t_n)^{2/r}t_n^{\mu}.
	\end{align*}
	The proof of the lemma is complete.
	 \end{proof}
\end{lemma}
\section{Stability result}\label{sec:sta}
\begin{lemma}(comparison principle, Corollary 4.1, \cite{submittied})
	\label{lemma2}
Set $\frac{1}{2}\le \sigma\le 1$ and  $0<\beta'<1$.	Let the temporal meshes satisfy $\lambda_1\tau^{\beta'}\le \frac{1}{2\Gamma(2-\beta')}$ and property P1 hold for the coefficients of $\delta_t^{\beta',*}v^n$. It holds that
\begin{align*}
	\begin{cases}
		(\delta_t^{\beta',*}-\lambda_1)v_1^j-\lambda_2v_1^{j-1}\le (\delta_t^{\beta',*}-\lambda_1)v_2^j-\lambda_2v_2^{j-1},\\
		v_1^0\le v_2^0,\ for\ all\ j\ge1,
	\end{cases}\Rightarrow v_1^j\le v_2^j,\ for\  j\ge0,
\end{align*}
 where $\lambda_1,\lambda_2\ge 0$.
\end{lemma}

\begin{lemma}
		\label{lemma3}
		Let $0<\beta'<1$ and $\sigma=1-\frac{\beta'}{2}$. Suppose that $\tau$ is sufficiently small and property P1 holds for the coefficients of $\delta_t^{\beta',*}v^n$. For $\gamma\in \mathbb{R}$ and $r\ge 1$, it holds that
		\begin{align}
			\label{eq3.2}
			\begin{cases}
				\delta_t^{\beta',*} v^j\lesssim (\tau_1/t_j)^{\gamma+1}\\
				\forall\ j\ge 1, v^0=0
			\end{cases}\Rightarrow  v^j\lesssim V(j,\gamma):= \mathcal{M}_{\gamma}(t_j)\tau_1 t_j^{\beta'-1}(\tau_1/t_j)^{\min\{0,\gamma\}},
		\end{align}
		where $\mathcal{M}_\gamma(t_j)=1+\ln(t_j/\tau_1)$ for $\gamma=0$ and $\mathcal{M}_\gamma(t_j)=1$ for $\gamma\neq0$.
		\begin{proof}
		 Set $\hat\gamma=\min\left\{\beta',\frac{1-\beta'}{r}\right\}$, and $\mathbb{B}(s;t_{p_k}):=\min\{(s/t_{p_k})t_{p_k}^{\beta'-1},s^{\beta'-1}\}$.
	We first prove the case when $\gamma\le \hat\gamma$.	We construct a barrier function $U_K(s)$ satisfying
		\begin{align*}
			U_K(s):=\sum_{k=0}^{K}w_k\mathbb{B}(s,t_{p_k}),
		\end{align*}
		where $p_k:=2^kp$,   $w_k:=(t_p/t_{p_k})^{\gamma}$ and $p\ge 2$ is a sufficiently large constant.
				
		It has already been estimated in Reference \cite{2025arXiv250612954K} that
		\begin{align}
			\Gamma(1-\beta')D_t^{\beta'}\mathbb B(s,t_{p_k})\ge \begin{cases}
				(1-\beta')(t_{p_k}/s)^{\beta'}s^{-1},\ for\ s>t_{p_k}>0,\\
				(1-\beta')^{-1}(s/t_{p_k})^{1-\beta'}t_{p_k}^{-1},\ for\ 0<s\le t_{p_k}.
			\end{cases}
		\end{align}		
	 Suppose that $t_{p_k}< t_j^*\le t_{p_{k+1}}$ and $0\le k\le K$. 	According to mesh condition\eqref{mesh}, there exists a constant $C_*$ such that $t_{p_k}/t_j^*\ge t_{p_k}/t_{p_{k+1}}\ge C_*2^{-r}$, which implies that
		\begin{align*}
			D_t^{\beta'}U_K(t_j^*)&\ge w_k D_t^{\beta'}\mathbb{B}(t_j^*,t_{p_k})\\
			&\ge w_k(\Gamma(1-\beta'))^{-1}(1-\beta')(t_{p_k}/t_j^*)^{\beta'}(t_j^*)^{-1}\\
			&\ge (C_*2^{-r})^{\beta'-\gamma}(\Gamma(1-\beta'))^{-1}(1-\beta')(t_p/t_j^*)^\gamma(t_j^*)^{-1}\\
			&\ge C \frac{1-\beta'}{\Gamma(1-\beta')}(t_p/t_j)^\gamma(t_j)^{-1}.
		\end{align*}		
	Next, we consider the remaining case $0<t_j^*\le t_p$. Using $t_j^*/t_p\ge t_1^*/t_p$, one has
	\begin{align*}
		D_t^{\beta'}U_K(t_j^*)&\ge w_0D_t^{\beta'}\mathbb{B}(t_j^*,t_p)\\&\ge (\Gamma(2-\beta'))^{-1}(t_j^*/t_p)^{1-\beta'}t_p^{-1}\\&=(\Gamma(2-\beta'))^{-1}(t_p/t_j^*)^\gamma(t_j^*)^{-1}(t_j^*/t_p)^{2+\gamma-\beta'}\\
		&\ge (\Gamma(2-\beta'))^{-1}(t_p/t_j^*)^\gamma (t_j^*)^{-1}(\sigma C_*p^{-r})^{\max\{0,\gamma-\beta'+2\}}.
	\end{align*}
	Combining the two cases above, we conclude that:
	\begin{align}
			D_t^{\beta'}U_K(t_j^*)&\ge C (t_p/t_j)^\gamma(t_j)^{-1},\ for\ 0<t_j^*\le t_{p_{K+1}}.
	\end{align}
	 Note that $\mathbb{B}(s,t_{p_k})$ is a linear function for $s\le t_{p_k}$, so $\delta_{t}^{\beta',*}U_K^j=D_t^{\beta'}U_K(t_j^*)$ for $j\le p$. Next we focus on estimate $\left|\delta_{t}^{\beta',*}U_K^j-D_t^{\beta'}U_K(t_j^*)\right|$ and ensure that it is less than $\frac{1}{2}D_t^{\beta'}U_K(t_j^*)$ for all $p<j\le p_{K+1}$.
	
	  According to inequalities \eqref{equation14} and \eqref{equation15}, we find that	for $s\in (t_{k-1},t_{k})$ and $p_n+1\le k\le j-1$,
	 \begin{align*}
	 	\left|(\Pi_{2,k}\mathbb{B}-\mathbb{B})(s,t_{p_n})\right|\le\frac{1}{6} (C_r+1)\tau_k^3 \left|\mathbb B'''(t_{k-1},t_{p_n})\right|\le\frac{1}{6} (C_r+1) (1-\beta')(2-\beta')(3-\beta')C_m^{4-\beta'}\tau_k^3t_{k}^{\beta'-4}.
	 \end{align*}
	 When $s\in (t_{j-1},t_{j}^*)$ and $p_n<j$, it holds that
	 \begin{align*}		\left|(\Pi_{2,j^*}\mathbb{B}-\mathbb{B})(s,t_{p_n})\right|\le \frac{1}{6}\tau_j^2(t_j^*-s)\left|\mathbb B'''(t_{j-1},t_{p_n})\right|\le\frac{1}{6}(1-\beta')(2-\beta')(3-\beta')C_m^{4-\beta'}\tau_j^2(t_j^*-s)t_{j}^{\beta'-4}.
	 \end{align*}
	  For convenience, we denote by $\mathbb B^{\pi}$ the interpolating function of $\mathbb B$. Determine the value of $J$ from the relation $p_J<j\le p_{J+1}$.
	If $s\in (t_{p_n},t_j^*)$, $B(s,t_{p_n})=s^{\beta'-1}$ is independent of $n$. So it holds that
	\begin{align*}
		\left|\delta_{t}^{\beta',*}U_K^j-D_t^{\beta'}U_K(t_j^*)\right|&\le \sum_{n=0}^{J}w_n\left|\frac{\beta'}{\Gamma(1-\beta')}\int_{t_{p_{n}-1}}^{t_j^*}(t_j^*-s)^{-\beta'-1}(\mathbb{B}^{\pi}-\mathbb{B})(s,t_{p_n})ds\right|\\
		&\le \sum_{n=0}^{J}w_n\frac{\beta'}{\Gamma(1-\beta')}\int_{t_{p_{n}-1}}^{t_j^*}(t_j^*-s)^{-\beta'-1}\left|(\mathbb{B}^{\pi}-\mathbb{B})(s,t_{p_n})\right|ds\\
		&=\sum_{n=0}^{J}w_n\frac{\beta'}{\Gamma(1-\beta')}\int_{t_{p_n}}^{t_j^*}(t_j^*-s)^{-\beta'-1}\left|(\mathbb{B}^{\pi}-\mathbb{B})(s,t_{p})\right|ds\\
		&+\sum_{n=0}^{J}w_n\frac{\beta'}{\Gamma(1-\beta')}\int_{t_{p_{n}-1}}^{t_{p_n}}(t_j^*-s)^{-\beta'-1}\left|(\mathbb{B}^{\pi}-\mathbb{B})(s,t_{p_n})\right|ds\\
		&:=G_1+G_2.
	\end{align*}
The integral term $G_1$ can be bounded by an integral that is independent of n; hence we have
\begin{equation}
		\begin{aligned}
			\label{equation20}
		G_1&=\sum_{n=0}^{J}w_n\frac{\beta'}{\Gamma(1-\beta')}\int_{t_{p_n}}^{t_j^*}(t_j^*-s)^{-\beta'-1}\left|(\mathbb{B}^{\pi}-\mathbb{B})(s,t_{p})\right|ds\\
		&\le \sum_{n=0}^{J}w_n\frac{\beta'}{\Gamma(1-\beta')}\int_{t_{p}}^{t_j^*}(t_j^*-s)^{-\beta'-1}\left|(\mathbb{B}^{\pi}-\mathbb{B})(s,t_{p})\right|ds\\
		&\le SUM_J C \frac{\beta'\Gamma(4-\beta')}{\Gamma(1-\beta')^2} \tau_1^{3/r}\max_{p+1\le k\le j}\{t_k^{\beta'+\frac{\beta'-3}{r}}\}\int_{t_{p}}^{t_j^*}(t_j^*-s)^{-\beta'-1}s^{-1-\beta'/r}\min\{1,\frac{t_j^*-s}{\sigma\tau_j}\}ds\\
		&\le SUM_J C \frac{\beta'\Gamma(4-\beta')}{\Gamma(1-\beta')^2} \max_{p+1\le k\le j}\{t_k^{\beta'+\frac{\beta'-3}{r}}\}(t_j^*)^{-\beta'-1}\left(\frac{\tau_1^{3/r}t_p^{-\beta'/r}+\tau_1^{\frac{3-\beta'}{r}}}{\beta'}+\frac{\tau_1^{\frac{3-\beta'}{r}}}{1-\beta'}\right).
	\end{aligned}
\end{equation}
Here, $SUM_J=\sum_{n=0}^{J}w_n$.	The derivation of the third inequality is omitted here; please refer to the proof of inequalities \eqref{equation16} and \eqref{equation17} for details.
Moreover, it follows that:
\begin{equation}
		\begin{aligned}
			\label{equation18}
		\max_{p+1\le k\le j}\{t_k^{\beta'+\frac{\beta'-3}{r}}\}(t_j^*)^{-\beta'-1}\tau_1^{3/r}t_p^{-\beta'/r}&\le \sigma^{-1-\beta'} (t_j)^{-1}(\tau_1/t_p)^{\frac{3}{r}}\max_{p+1\le k\le j}\left[\left(\frac{t_k}{t_j}\right)^{\beta'}\left(\frac{t_p}{t_k}\right)^{\frac{3-\beta'}{r}}\right]\\&\le \sigma^{-1-\beta'} (t_j)^{-1}(\tau_1/t_p)^{\frac{3}{r}}(t_p/t_j)^{\min\{\beta',\frac{3-\beta'}{r}\}},
	\end{aligned}
\end{equation}
	and
		\begin{equation}
			\begin{aligned}
				\label{equation19}
				\max_{p+1\le k\le j}\{t_k^{\beta'+\frac{\beta'-3}{r}}\}(t_j^*)^{-\beta'-1}\tau_1^{\frac{3-\beta'}{r}}&\le \sigma^{-1-\beta'} (t_j)^{-1}\max_{p+1\le k\le j}\left[\left(\frac{t_k}{t_j}\right)^{\beta'}\left(\frac{\tau_1}{t_k}\right)^{\frac{3-\beta'}{r}}\right]\\&\le \sigma^{-1-\beta'} (t_j)^{-1}(\tau_1/t_j)^{\min\{\beta',\frac{3-\beta'}{r}\}}.
			\end{aligned}
		\end{equation}
	Substituting \eqref{equation18} and \eqref{equation19} into \eqref{equation20}, we obtain the estimate for $G_1$,
	\begin{align*}
		G_1&\le SUM_J C \frac{\beta'\Gamma(4-\beta')}{\Gamma(1-\beta')^2}\left[\frac{(t_j)^{-1}}{\beta'}(\tau_1/t_p)^{\frac{3}{r}}(t_p/t_j)^{\min\{\beta',\frac{3-\beta'}{r}\}}+\frac{(t_j)^{-1}}{\beta'(1-\beta')}(\tau_1/t_j)^{\min\{\beta',\frac{3-\beta'}{r}\}}\right]\\
		&\le SUM_J C(t_j)^{-1} \frac{\beta'\Gamma(4-\beta')}{\Gamma(1-\beta')^2}(\tau_1/t_p)^{\min\{\beta',\frac{3-\beta'}{r}\}}\bigg[\frac{1}{\beta'}(\tau_1/t_p)^{\frac{3}{r}-\min\{\beta',\frac{3-\beta'}{r}\}}(t_p/t_j)^{\min\{\beta',\frac{3-\beta'}{r}\}}\\
&\qquad +\frac{1}{\beta'(1-\beta')}(t_p/t_j)^{\min\{\beta',\frac{3-\beta'}{r}\}}\bigg]\\
		&\le SUM_J C(t_j)^{-1} \frac{\beta'\Gamma(4-\beta')}{\Gamma(1-\beta')^2}(\tau_1/t_p)^{\min\{\beta',\frac{3-\beta'}{r}\}}\left[\frac{1}{\beta'}(t_p/t_j)^{\min\{\beta',\frac{3-\beta'}{r}\}}+\frac{1}{\beta'(1-\beta')}(t_p/t_j)^{\min\{\beta',\frac{3-\beta'}{r}\}}\right]\\
		&\le SUM_J C(t_j)^{-1} \frac{(2-\beta')\Gamma(4-\beta')}{\Gamma(1-\beta')\Gamma(2-\beta')}(\tau_1/t_p)^{\min\{\beta',\frac{3-\beta'}{r}\}}(t_p/t_j)^{\min\{\beta',\frac{3-\beta'}{r}\}}.
	\end{align*}
	Next, we estimate $G_2$.  For $s\in (t_{p_n-1},t_{p_n})$,  one has $\left|(\Pi_{2,p_n}\mathbb{B}-\mathbb{B})(s,t_{p_n})\right|\le \left|\Pi_{2,p_n}(\mathbb{B}(s,t_{p_n})-\mathbb{B}(t_{p_n},t_{p_n}))\right|+\left|\mathbb{B}(s,t_{p_n})-\mathbb{B}(t_{p_n},t_{p_n})\right|\le C \max_{t_{p_n-1}\le s \le t_{p_n+1}}\left|\mathbb{B}(s,t_{p_n})-\mathbb{B}(t_{p_n},t_{p_n})\right|\le C\tau_{p_n}t_{p_n}^{\beta'-2}$, using the fact that
	\begin{align*}
		 \max_{t_{p_n-1}\le s\le t_{p_n}}\left|\Pi_{2,p_n}(\mathbb{B}(s,t_{p_n})-\mathbb{B}(t_{p_n},t_{p_n}))\right|\le C \max_{t_{p_n-1}\le s\le t_{p_n+1}}\left|\mathbb{B}(s,t_{p_n})-\mathbb{B}(t_{p_n},t_{p_n})\right|.
	\end{align*}
	Since $(\tau_{p_n}/t_{p_n})t_{p_n}^{\beta'-1}\le C (\tau_1/t_{p_n})^{1/r} t_{p_n}^{\beta'-1}$, it holds that
	\begin{equation}
		\begin{aligned}
			\label{equation21}
			G_2&=\sum_{n=0}^{J}w_n\frac{\beta'}{\Gamma(1-\beta')}\int_{t_{p_{n}-1}}^{t_{p_n}}(t_j^*-s)^{-\beta'-1}\left|(\mathbb{B}^{\pi}-\mathbb{B})(s,t_{p_n})\right|ds\\
			&\le C \sum_{n=0}^{J}w_n\frac{\beta'}{\Gamma(1-\beta')}\int_{t_{p_{n}-1}}^{t_{p_n}}(t_j^*-s)^{-\beta'-1}(\tau_1/t_{p_n})^{1/r} t_{p_n}^{\beta'-1}ds\\
			&\le C \sum_{n=0}^{J}w_n\frac{\beta'}{\Gamma(1-\beta')}\tau_1^{1/r}\max_{p\le k\le j}\{(t_k)^{\frac{\beta'-1}{r}+\beta'}\}\int_{t_{p_{n}-1}}^{t_{p_n}}(t_j^*-s)^{-\beta'-1}s^{-1-\frac{\beta'}{r}}ds.
		\end{aligned}
	\end{equation}
Due to $p_J<j$, we have $t_{p-1}\le t_{p_n-1}\le t_{p_n}\le t_{j-1}$; consequently, the integral of $G_2$ can be enlarged to
\begin{equation}
		\begin{aligned}
			\label{equation22}
		\int_{t_{p_{n}-1}}^{t_{p_n}}(t_j^*-s)^{-\beta'-1}s^{-1-\frac{\beta'}{r}}ds&\le \int_{t_{p-1}}^{t_{j-1}}(t_j^*-s)^{-\beta'-1}s^{-1-\frac{\beta'}{r}}ds\\
		&=(t_j^*)^{-\beta'-1-\frac{\beta'}{r}}\int_{t_{p-1}/t_j^*}^{t_{j-1}/t_j^*}(1-s)^{-\beta'-1}s^{-1-\frac{\beta'}{r}}ds\\
		&\le C(t_j^*)^{-\beta'-1-\frac{\beta'}{r}}\frac{1}{\beta'}\left[(t_{p-1}/t_j^*)^{-\frac{\beta'}{r}}+(\sigma\tau_j/t_j^*)^{-\beta'}\right]\\
		&\le C(t_j)^{-\beta'-1-\frac{\beta'}{r}}\frac{1}{\beta'}\left[(t_{p}/t_j)^{-\frac{\beta'}{r}}+(\tau_j/t_j)^{-\beta'}\right].
	\end{aligned}
\end{equation}
Inserting the result of \eqref{equation22} into \eqref{equation21}, we obtain
\begin{align*}
	G_2 &\le C SUM_J\frac{1}{\Gamma(1-\beta')}\tau_1^{1/r}\max_{p\le k\le j}\{(t_k)^{\frac{\beta'-1}{r}+\beta'}\}(t_j)^{-\beta'-1}\left[(t_{p})^{-\frac{\beta'}{r}}+(\tau_1)^{-\beta'/r}\right]\\
	&=C SUM_J\frac{1}{\Gamma(1-\beta')}(t_j)^{-1}\max_{p\le k\le j}\left\{\left(\frac{t_k}{t_j}\right)^{\beta'}(t_k)^{\frac{\beta'-1}{r}}\right\}\left[\tau_1^{1/r}(t_{p})^{-\frac{\beta'}{r}}+(\tau_1)^{\frac{1-\beta'}{r}}\right]\\
	&\le C SUM_J\frac{1}{\Gamma(1-\beta')}(t_j)^{-1}(\tau_1/t_p)^{\min\{\beta',\frac{1-\beta'}{r}\}}\left[(\tau_1/t_p)^{1/r-\min\{\beta',\frac{1-\beta'}{r}\}}(t_p/t_j)^{\min\{\beta',\frac{1-\beta'}{r}\}}+(t_p/t_j)^{\min\{\beta',\frac{1-\beta'}{r}\}}\right]\\
	&\le C SUM_J\frac{1}{\Gamma(1-\beta')}(t_j)^{-1}(\tau_1/t_p)^{\min\{\beta',\frac{1-\beta'}{r}\}}(t_p/t_j)^{\min\{\beta',\frac{1-\beta'}{r}\}}.
\end{align*}
In reference \cite{2025arXiv250612954K}, it has been established that for any $K\ge 0$,
\begin{equation}
	\begin{aligned}
		SUM_K\le C_{\gamma,r}\begin{cases}
			1,\ &if\ \gamma>0,\\
			{1+\ln(t_{p_K}/t_p)},\ &if\ \gamma=0,\\
			(t_p/t_{p_K})^{\gamma},\ &if\ \gamma<0.
		\end{cases}
	\end{aligned}
\end{equation}
 For $SUM_J$, using $t_{p_J}/t_p\le t_{j}/t_p$ and $(1+\ln(x))x^{-\hat\gamma}\le (\hat\gamma)^{-1}$ for $x\ge 1$, one has
 \begin{align*}
 	SUM_J(t_p/t_j)^{\hat\gamma}\le \begin{cases}
 		C_{\gamma,r}\frac{\beta'(r-1)+1}{\beta'(1-\beta')}(t_p/t_j)^\gamma\ if\ \gamma=0\\
 		C_{\gamma,r}(t_p/t_j)^\gamma\ if\ \gamma\neq0.
 	\end{cases}
 \end{align*}
Combining the estimates for $G_1$ and $G_2$ yields
\begin{equation}
	\begin{aligned}
		\left|\delta_{t}^{\beta',*}U_K^j-D_t^{\beta'}U_K(t_j^*)\right|&\le  C SUM_J\frac{1}{\Gamma(1-\beta')}(t_j)^{-1}(\tau_1/t_p)^{\min\{\beta',\frac{1-\beta'}{r}\}}(t_p/t_j)^{\min\{\beta',\frac{1-\beta'}{r}\}}\\
		&\le C \frac{1}{\beta'\Gamma(2-\beta')}(t_j)^{-1}(\tau_1/t_p)^{\hat\gamma}(t_p/t_j)^{\gamma},
	\end{aligned}
\end{equation}
for $j>p$. We can choose a sufficiently large $p$ such that the following holds:
$	\left|\delta_{t}^{\beta',*}U_K^j-D_t^{\beta'}U_K(t_j^*)\right|\le \frac{1}{2}D_t^{\beta'}U_K(t_j^*)$  for $p<j\le p_{K+1}$.
 Therefore, when $\tau$ is sufficiently small, there exists a barrier function satisfying the following conditions:
 \begin{align}
 	\label{equation23}
 	\delta_{t}^{\beta',*}U_K^j\ge C (t_j)^{-1}(\tau_1/t_j)^{\gamma},\ and\ U_K^j\le SUM_K(t_j)^{\beta'-1},\ for\ 1\le j\le p_{K+1}.
 \end{align}
It is worth noting that the constant $C$ in inequality \ref{equation23} is independent of $K$; $C$ depends on $\beta'$, $p$, $r$ and $\gamma$.

According to the comparison principle (Lemma \ref{lemma2}), we thus obtain the estimate for $v^j$. When $0<j\le p_1$, it holds that there exists a positive constant $d_1$ such that  $\delta_t^{\beta',*}v^k\le d_1\tau_1\delta_{t}^{\beta',*}U_0^k$ for $1\le k\le j$, thus $v^j\le d_1 SUM_0\tau_1 (t_j)^{\beta'-1}\le  d_1 C_{\gamma,r}\tau_1 (t_j)^{\beta'-1}\le d_1 C_{\gamma,r}V(j,\gamma)$. When $p_1< j\le p_2$, we have $\delta_t^{\beta',*}v^k\le d_1\tau_1\delta_{t}^{\beta',*}U_1^k$ for $1\le k\le j$, and
\begin{align*}
	v^j\le d_1 SUM_1\tau_1 (t_j)^{\beta'-1}&\le  d_1 C_{\gamma,r}\tau_1 (t_j)^{\beta'-1}\begin{cases}
		1,\ &if\ \gamma>0,\\
		1+\ln(t_{p_1}/t_p),\ &if\ \gamma=0,\\
		(t_p/t_{p_1})^{\gamma},\ &if\ \gamma<0,
	\end{cases}\\
	&\le d_1 C_{\gamma,r}\tau_1 (t_j)^{\beta'-1}\begin{cases}
	1,\ &if\ \gamma>0,\\
	1+\ln(t_{j}/\tau_1),\ &if\ \gamma=0,\\
	(\tau_1/t_{j})^{\gamma},\ &if\ \gamma<0,
	\end{cases}\\
	&=d_1 C_{\gamma,r}V(j,\gamma).
\end{align*}
Proceeding in the same manner, yields $v^j\le d_1C_{\gamma,r}V(j,\gamma)$ for $j\ge 1$.

The above is the proof under the assumption that  $\gamma\le\hat\gamma$. If $\gamma>\hat\gamma$, for $j\ge 1$, we will have
\begin{align*}
	\delta_t^{\beta',*}v^j\lesssim (\tau_1/t_j)^{\gamma+1}\le (\tau_1/t_j)^{\hat\gamma+1}\ and\ 	v^j\lesssim V(j,\hat\gamma)=\tau_1t_j^{\beta'-1}=V(j,\gamma).
\end{align*}
Therefore, the proof of the lemma is complete.
		\end{proof}
\end{lemma}

\begin{lemma}
	\label{lemma6}
	Let $\sigma=1-\frac{\beta'}{2}$. Assume that $\tau$ is sufficiently small and property P1 holds for the coefficients of $\delta_t^{\beta',*}v^n$ with $0<\beta'<1$. $\lambda_1$ and $\lambda_2$ are non-negative constants and at least one of them is positive.  For $\gamma\in \mathbb{R}$ and $r\ge 1$, it holds that
	\begin{align}
		\begin{cases}
			(\delta_t^{\beta',*}-\lambda_1)v^j-\lambda_2 v^{j-1}\lesssim (\tau_1/t_j)^{\gamma+1}\\
			\forall\ j\ge 1, v^0=0
		\end{cases}\Rightarrow  v^j\lesssim V(j,\gamma):= \mathcal{M}_{\gamma}(t_j)\tau_1 t_j^{\beta'-1}(\tau_1/t_j)^{\min\{0,\gamma\}},
	\end{align}
	where $\mathcal{M}_\gamma(t_j)=1+\ln(t_j/\tau_1)$ for $\gamma=0$ and $\mathcal{M}_\gamma(t_j)=1$ for $\gamma\neq0$.
\begin{proof}
From Lemma \ref{lemma3}, there exist $\delta_t^{\beta',*}\xi_\gamma^j=(\tau_1/t_j)^{\gamma+1}$ and $\delta_t^{\beta',*}\xi_{\gamma^*}^j=(\tau_1/t_j)^{\gamma^*+1}$ for $\gamma\in\mathbb{R}$ and $j\ge 1$ such that the following inequalities hold for $j\ge 1$:
\begin{equation}
	\begin{aligned}
		\label{equation25}
		\xi_\gamma^j\le C_1\mathcal{M}_\gamma(t_j)\tau_1^{\beta'}(\tau_1/t_j)^{1+\gamma^*},\
		\xi_{\gamma^*}^j\le C_2\tau_1^{\beta'}(\tau_1/t_j)^{1+\gamma^*-\beta'},\\
		(\delta_t^{\beta',*}-\lambda_1)	\xi_\gamma^j-\lambda_2 	\xi_\gamma^{j-1}\ge (\tau_1/t_j)^{\gamma+1}- C_\gamma \mathcal{M}_\gamma(T)\tau_1^{\beta'}(\tau_1/t_j)^{1+\gamma^*},\\
		(\delta_t^{\beta',*}-\lambda_1)	\xi_{\gamma^*}^j-\lambda_2 	\xi_{\gamma^*}^{j-1}\ge (\tau_1/t_j)^{\gamma^*+1}(1-C_\gamma t_j^{\beta'}),
	\end{aligned}
\end{equation}
where $\gamma^*=\min\{0,\gamma\}-\beta'<0$. We construct the following barrier function
\begin{align*}
	S^j=\xi_\gamma^j+\mathcal{M}_\gamma(T)\bar b\tau_1^{\beta'}\xi_{\gamma^*}^j+\mathcal{M}_\gamma(T)(\bar b)^2\tau_1^{1+\min\{0,\gamma\}}B^j,
\end{align*}
where $\bar b\ge C_\gamma\max\{2, c_1^{-1-\gamma^*+\beta'},T^{-1-\gamma^*+\beta'}\}$ is a sufficiently large constant, and $c_1\le (2C_\gamma)^{-\frac{1}{\beta'}}$.  The subsequent derivation requires only minor modifications to the proof of Theorem 4.1 in \cite{submittied}. Hence,
\begin{align*}
	&(\delta_t^{\beta',*}-\lambda_1)	(\xi_\gamma^j+\mathcal{M}_\gamma(T)\bar b\tau_1^{\beta'}\xi_{\gamma^*}^j)-\lambda_2 (	\xi_\gamma^{j-1}+\mathcal{M}_\gamma(T)\bar b\tau_1^{\beta'}\xi_{\gamma^*}^{j-1}) \\
&\ge (\tau_1/t_j)^{\gamma+1}-\mathcal{M}_\gamma(T)\left[C_\gamma \tau_1^{\beta'}(\tau_1/t_j)^{1+\gamma^*}-\bar b\tau_1^{\beta'}(\tau_1/t_j)^{\gamma^*+1}(1-C_\gamma t_j^{\beta'})\right]\\
	&\ge (\tau_1/t_j)^{\gamma+1}-\mathcal{M}_\gamma(T)(\bar b)^2\begin{cases}
		0,\ &t_j<c_1,\\
		\tau_1^{1+\min\{\gamma,0\}},\ &t_j\ge c_1.
	\end{cases}
\end{align*}
Set $c_0=\min\{\frac{c_1}{3},\frac{(2(\lambda_1+\lambda_2)\Gamma(2-\beta'))^{-1/\beta'}}{3}\}$, and $\tau\le c_0/2$. We choose $t_m$ such that $\frac{1}{2}c_0\le t_m\le c_0$ if $T\ge \frac{1}{2}c_0$, $t_m=T$ otherwise. This ensures that $t_m+\frac{5}{4}c_0\le c_1$ and $t_j^*\ge t_m+c_0$ when $t_j\ge c_1$.  Lemma 4.2 in \cite{submittied} implies that there exists a function $B^j$ satisfies that $B^j=0$ if $0\le j\le m$,  $0\le B^j\le C_3$ if $j>m$, and
\begin{align*}
	(\delta_t^{\beta',*}-\lambda_1)B^j-\lambda_2B^{j-1}\ge \begin{cases}
		0,\ &t_j^*<t_m+c_0,\\
		1,\ &t_j^*\ge t_m+c_0,
	\end{cases}
	\ for\ j\ge 1.
\end{align*} Thus, we obtain
\begin{align*}
	(\delta_t^{\beta',*}-\lambda_1)S^j-\lambda_2S^{j-1}\ge (\tau_1/t_j)^{\gamma+1}.
\end{align*}
For $j> m$, it holds that
\begin{equation}
	\begin{aligned}
		\label{equation24}
		\bar b\tau_1^{\beta'}\xi_{\gamma^*}^j+(\bar b)^2\tau_1^{1+\min\{0,\gamma\}}B^j&\le C_2\bar b \tau_1^{\beta'}t_j^{\beta'}(\tau_1/t_j)^{1+\gamma^*}+C_3(\bar b)^2\tau_1^{\beta'}t_j^{\beta'}(\tau_1/t_j)^{1+\gamma^*} t_j^{1+\gamma^*-\beta'}\\
		&\le \max_{j>m}( C_2\bar b+C_3(\bar b)^2 t_j^{1+\gamma^*-\beta'})\tau_1^{\beta'}t_j^{\beta'}(\tau_1/t_j)^{1+\gamma^*}\\
		&\le  \left(C_2\bar b +C_3(\bar b)^2 \max\{T^{1+\gamma^*-\beta'},(c_0/2)^{1+\gamma^*-\beta'}\}\right)\tau_1^{\beta'}t_j^{\beta'}(\tau_1/t_j)^{1+\gamma^*}:=C_4\tau_1^{\beta'}t_j^{\beta'}(\tau_1/t_j)^{1+\gamma^*}.
	\end{aligned}
\end{equation}
For $j\le m$, $B^j=0$ yields that $\bar b\tau_1^{\beta'}\xi_{\gamma^*}^j+(\bar b)^2\tau_1^{1+\min\{0,\gamma\}}B^j\le C_4\tau_1^{\beta'}t_j^{\beta'}(\tau_1/t_j)^{1+\gamma^*}$ holds as well.
Note that the above observations \eqref{equation24} and \eqref{equation25} yields, for $j\ge 1$ and $\gamma=0$,
\begin{align*}
	S^j&\le (C_1\mathcal{M}_\gamma(t_j)+ C_4\mathcal{M}_\gamma(T)t_j^{\beta'})\tau_1^{\beta'}(\tau_1/t_j)^{1+\gamma^*}\\
	&= \left[C_1\mathcal{M}_\gamma(t_j)+C_4(1+\ln(t_j/\tau_1)+\ln(T/t_j))t_j^{\beta'}\right]\tau_1^{\beta'}(\tau_1/t_j)^{1+\gamma^*}\\
	&\le \left[C_1\mathcal{M}_\gamma(t_j)+C_4T^{\beta'}(1+\ln(t_j/\tau_1))+C_4\ln(T/t_j)t_j^{\beta'}\right]\tau_1^{\beta'}(\tau_1/t_j)^{1+\gamma^*}\\
	&\le \left(C_1+\frac{1+\beta'}{\beta'}C_4T^{\beta'}\right)\mathcal{M}_\gamma(t_j)\tau_1^{\beta'}(\tau_1/t_j)^{1+\gamma^*}.
\end{align*}
The last inequality follows from the fact that  $\ln(x)\le \frac{1}{\beta'}x^{\beta'}$ for $x\ge 1$.  If $\gamma\neq 0$, it is obvious that $S^j\lesssim\mathcal{M}_\gamma(t_j)\tau_1^{\beta'}(\tau_1/t_j)^{1+\gamma^*}$.
Applying the comparison principle, the proof of this lemma can be concluded.
\end{proof}
\end{lemma}

\section{Convergence and stability analysis}\label{sec:con}
In this section, we present a key inequality that will be used for the convergence analysis. To this end, we first introduce some notation for arbitrary grid functions $U,V$ $\in \Pi_h=\{v|v\ is\ \ a\ grid\ function\ defined\ in\ \bar\Omega_h,\ and\ v=0\ on\  \partial\Omega_h\}$.

We define
\begin{gather*}
	(U,V)=h^2\sum_{i=1}^{M-1}\sum_{j=1}^{M-1}{U_{i,j}V_{i,j}},\ \|U\|=\sqrt{(U,U)},\\
	\delta_xU_{i-\frac{1}{2},j}=\frac{U_{i,j}-U_{i-1,j}}{h},\ \delta_yU_{i,j-\frac{1}{2}}=\frac{U_{i,j}-U_{i,j-1}}{h},\\
	(U,V)_x=h^2\sum_{i=1}^{M}\sum_{j=1}^{M-1}{\delta_xU_{i-\frac{1}{2},j}\delta_xV_{i-\frac{1}{2},j}},\\
	(U,V)_y=h^2\sum_{i=1}^{M-1}\sum_{j=1}^{M}{\delta_yU_{i,j-\frac{1}{2}}\delta_yV_{i,j-\frac{1}{2}}},\\
	(\nabla_hU,\nabla_hV)=(U,V)_x+(U,V)_y.
\end{gather*}
One can easily verify that $-(U,\Delta_hV)=(\nabla_hU,\nabla_hV)$ for  $U,V\in \Pi_h$.
\begin{lemma}(Lemma 2.1, \cite{MR2588905})
	\label{lemma4}
	For grid function $U\in\Pi_h$, one has
	\begin{align}
		\|U\|\le C\|\nabla_h U\|.
	\end{align}
\end{lemma}

\begin{lemma}
	\label{keylemma}
	Assume that the coefficients of $\delta_{t}^{\beta',*}v^n$ satisfy properties P1 and P2 with $0<\beta'<1$, and $1\ge\sigma\ge \frac{1}{2}$. For $n=1,...,N$, one has \begin{align}
		(\delta_{t}^{\beta',*}v_1^n,v_1^{n,*})+(\delta_{t}^{\beta',*}v_2^n,v_2^{n,*})\ge \left[\sigma \xi^n+(1-\sigma)\xi^{n-1}\right]\delta_t^{\beta',*}\xi^n.
	\end{align}
	where  $\{v_1^n\}_{n=0}^N$and $\{v_2^n\}_{n=0}^N\in \Pi_h$,  $\xi^n=\sqrt{\|v_1^n\|^2+\|v_2^n\|^2}$.
	\begin{proof}
		We introduce a novel discrete representation:  $\delta_t^{\beta,*}v^n=\sum_{k=0}^{n}p_{n,k}v^k$, where the coefficients are defined by  $p_{n,k}=g_{n,k}-g_{n,k+1}$. Here, we set $g_{n,0}=g_{n,n+1}=0$.
		If property P1 holds, then $p_{n,n}>0$, whereas $p_{n,k}< 0$ for all $0\le k\le n-1$.
		Let $\delta_{t}^{\beta',*}v_1^n$ take the inner product with $v_1^n$ and $v_1^{n-1}$ respectively, yielding:
		\begin{align*}
			(\delta_{t}^{\beta',*}v_1^n,v_1^n)=p_{n,n}(v_1^n,v_1^n)+p_{n,n-1}(v_1^{n-1},v_1^n)+\sum_{j=0}^{n-2}p_{n,j}(v_1^j,v_1^n),\\			(\delta_{t}^{\beta',*}v_1^n,v_1^{n-1})=p_{n,n}(v_1^n,v_1^{n-1})+p_{n,n-1}(v_1^{n-1},v_1^{n-1})+\sum_{j=0}^{n-2}p_{n,j}(v_1^j,v_1^{n-1}).
		\end{align*}
		It is easy to get
		\begin{align*}
			\sigma	(\delta_{t}^{\beta',*}v_1^n,v_1^n)+(1-\sigma)(\delta_{t}^{\beta',*}v_1^n,v_1^{n-1})&=\left[\sigma p_{n,n-1}+(1-\sigma)p_{n,n}\right](v_1^n,v_1^{n-1})\\
			&+\sigma\left[p_{n,n}(v_1^n,v_1^n)+\sum_{j=0}^{n-2}p_{n,j}(v_1^j,v_1^n)\right]+(1-\sigma)\sum_{j=0}^{n-1}p_{n,j}(v_1^j,v_1^{n-1}).
		\end{align*}
		In the same manner, $(\delta_{t}^{\beta',*}v_2^n,v_2^{n,*})$ will yield similar results. Furthermore, by adding the two equations, we obtain
		\begin{align*}
			(\delta_{t}^{\beta',*}v_1^n,v_1^{n,*})+(\delta_{t}^{\beta',*}v_2^n,v_2^{n,*})&=\left[\sigma p_{n,n-1}+(1-\sigma)p_{n,n}\right]\left[(v_1^n,v_1^{n-1})+(v_2^n,v_2^{n-1})\right]\\
			&+\sigma\left\{p_{n,n}\left[(v_1^n,v_1^n)+(v_2^n,v_2^n)\right]+\sum_{j=0}^{n-2}p_{n,j}\left[(v_1^j,v_1^n)+(v_2^j,v_2^n)\right]\right\}\\
			&+(1-\sigma)\sum_{j=0}^{n-1}p_{n,j}\left[(v_1^j,v_1^{n-1})+(v_2^j,v_2^{n-1})\right].
		\end{align*}
		
		Employing property P2, for $n\ge 2$, one has \begin{align*}
			\sigma p_{n,n-1}+(1-\sigma)p_{n,n}&=\sigma (g_{n,n-1}-g_{n,n})+(1-\sigma)g_{n,n}\\&=\sigma g_{n,n-1}+(1-2\sigma)g_{n,n}\\&\le 0.
		\end{align*}
		In the alternative case where $n=1$, by utilizing the condition  $2\sigma\ge1$, we obtain:
		\begin{align*}
			\sigma p_{1,0}+(1-\sigma)p_{1,1}&=\sigma (-g_{1,1})+(1-\sigma)g_{1,1}\\&=(1-2\sigma)g_{1,1}\\&\le 0.
		\end{align*}
		By Cauchy-Schwarz inequality, for arbitrary grid function $\eta_1$, $\eta_2$, $\eta_3$ and $\eta_4$, one has		
		\begin{align}
			\label{C-S}
			(\eta_1,\eta_2)+	(\eta_3,\eta_4)\le \|\eta_1\|\|\eta_2\|+\|\eta_3\|\|\eta_4\|\le \sqrt{\|\eta_1\|^2+\|\eta_3\|^2}\sqrt{\|\eta_2\|^2+\|\eta_4\|^2}.
		\end{align}		 Using property P1 and inequality \eqref{C-S}, for $n\ge 1$, we have
		\begin{align*}
			(\delta_{t}^{\beta',*}v_1^n,v_1^{n,*})+(\delta_{t}^{\beta',*}v_2^n,v_2^{n,*})&\ge\left[\sigma p_{n,n-1}+(1-\sigma)p_{n,n}\right]\xi^n\xi^{n-1}\\
			&+\sigma\left[p_{n,n}(\xi^n)^2+\sum_{j=0}^{n-2}p_{n,j}\xi^j\xi^{n}\right]+(1-\sigma)\sum_{j=0}^{n-1}p_{n,j}\xi^j\xi^{n-1}\\
			&=\sigma\left[p_{n,n}(\xi^n)^2+\sum_{j=0}^{n-1}p_{n,j}\xi^j\xi^{n}\right]+(1-\sigma)\sum_{j=0}^{n}p_{n,j}\xi^j\xi^{n-1}\\
			&=\left[\sigma \xi^n+(1-\sigma)\xi^{n-1}\right]\delta_t^{\beta',*}\xi^n,
		\end{align*}
		where $\xi^n=\sqrt{\|v_1^n\|^2+\|v_2^n\|^2}$.	This proof is completed.
	\end{proof}
\end{lemma}

Suppose that $(p^n,\bar{u}^n)$ and $(P^n,\bar{U}^n)$ is the solution of \eqref{equation4} and \eqref{equation5}, respectively. By introducing $\tilde{p}^n=p^n-P^n$ and $\tilde{u}^n=u^n-U^n=\bar{u}^n-\bar{U}^n$ to denote errors, one has
\begin{subequations}
	\label{equation6}
	\begin{align}
			\label{equation6a}
		& \delta_{t}^{\beta,*} \tilde p_{i,j}^{n}-\nu\Delta_h \tilde u_{i,j}^{n,*}+F(u_{i,j}^{n,*})-F(U_{i,j}^{n,*})=(R_1)_{i,j}^n,\ for\ (x_i,y_j)\in \Omega_h,1\le n\le N,\\
			\label{equation6b}
		&\delta_{t}^{\beta,*}\tilde u_{i,j}^n=\tilde p_{i,j}^{n,*}+(R_2)_{i,j}^n,\ for\ (x_i,y_j)\in \Omega_h,1\le n\le N,\\
		&\tilde u_{i,j}^0=0,\ \tilde p_{i,j}^0=0,\ for\ (x_i,y_j)\in \Omega_h,\\
		&\tilde u_{i,j}^n=0,\ \tilde p_{i,j}^n=0,\ for\ (x_i,y_j)\in\partial\Omega_h, 0\le n\le N.
	\end{align}
\end{subequations}

\begin{theorem}(Convergence analysis)
	\label{theorem1}
	If $\sigma=1-\frac{\beta}{2}$ and properties P1 and P2 hold, the scheme \eqref{equation5} is convergent with sufficiently small $\tau$  for $r\ge 1$. The error estimates are as follows:
	\begin{align}
		\|\tilde{p}^n\|+\|\nabla_h \tilde{u}^n\|\lesssim\begin{cases}
			\tau_1t_n^{\beta-1}+h^2t_n^\beta,\quad & if\ 1\le r\le 2/(\beta+1),\\
			\tau_1t_n^{\beta-1}+\tau_1^{2/r}t_n^{2\beta-2/r}+h^2t_n^\beta, \quad   &if\ 2/(\beta+1)< r <3-\beta,\\
			(1+\ln(t_n/\tau_1))\tau_1t_n^{\beta-1}+\tau_1^{2/r}t_n^{2\beta-2/r}+h^2t_n^\beta, \quad   &if\ r =3-\beta,\\
			\tau_1^{\frac{3-\beta}{r}}t_n^{\beta-\frac{3-\beta}{r}}+\tau_1^{2/r}t_n^{2\beta-2/r}+h^2t_n^\beta, \quad &if\ r>3-\beta.
		\end{cases}
	\end{align}
	\begin{proof}
		Taking the inner product on the both sides of \eqref{equation6a} and \eqref{equation6b} with $\tilde{p}^{n,*}$ and $-\nu\Delta_h\tilde{u}^{n,*}$, for $n\ge 1$, it is easy to get
		\begin{align*}
			 (\delta_{t}^{\beta,*} \tilde p^{n},\tilde p^{n,*})-(\nu\Delta_h \tilde u^{n,*},\tilde p^{n,*})&=-(F(u^{n,*})-F(U^{n,*}),\tilde p^{n,*})+((R_1)^n,\tilde p^{n,*}),\\
			-(\delta_{t}^{\beta,*}\tilde u^n,\nu\Delta_h\tilde{u}^{n,*})&=-(\tilde p^{n,*},\nu\Delta_h\tilde{u}^{n,*})-((R_2)^n,\nu\Delta_h\tilde{u}^{n,*}).
		\end{align*}
Adding the above two equations, one has
\begin{equation}
	\begin{aligned}
		\label{equation7}
		(\delta_{t}^{\beta,*} \tilde p^{n},\tilde p^{n,*})+\nu(\delta_{t}^{\beta,*}\nabla_h\tilde u^n,\nabla_h\tilde{u}^{n,*})&=-(F(u^{n,*})-F(U^{n,*}),\tilde p^{n,*})\\
		&+((R_1)^n,\tilde p^{n,*})-((R_2)^n,\nu\Delta_h\tilde{u}^{n,*}).
	\end{aligned}		
\end{equation}
		Then for all $1\le n\le N$, utilizing Lemma \ref{lemma4} and Lagrange mean value theorem, it holds that
	\begin{equation}
			\begin{aligned}
				\label{equation28}
			\|F(u^{n,*})-F(U^{n,*})\|&=\|f(u^{n-1})+\sigma f'(u^{n-1})(u^n-u^{n-1})-\left(f(U^{n-1})+\sigma f'(U^{n-1})(U^n-U^{n-1})\right)\|\\
			&\le\|f(u^{n-1})-f(U^{n-1})\|+\sigma\| f'(u^{n-1})(u^n-u^{n-1})- f'(U^{n-1})(U^n-U^{n-1})\|\\
			&\le \|f(u^{n-1})-f(U^{n-1})\|+\sigma\|(f'(u^{n-1})-f'(U^{n-1}))u^n\|+\sigma\|f'(U^{n-1})(u^n-U^{n})\|\\
			&+\sigma\|(f'(u^{n-1})-f'(U^{n-1}))u^{n-1}\|+\sigma\|f'(U^{n-1})(u^{n-1}-U^{n-1})\|\\&\le K(\|\tilde u^{n}\|+\|\tilde u^{n-1}\|)\le K_1(\|\nabla_h \tilde u^{n}\|+\|\nabla_h\tilde u^{n-1}\|),
		\end{aligned}
	\end{equation}
Here, both $K$ and $K_1$ are independent of $n$ and step sizes. According to Lemma \ref{keylemma}, The left-hand side term of equation \eqref{equation7} satisfies:
\begin{equation}
	\begin{aligned}
		\label{equation8}
		the\ left\ge \left[\sigma\sqrt{\|\tilde{p}^n\|^2+\nu\|\nabla_h\tilde{u}^n\|^2}+(1-\sigma)\sqrt{\|\tilde{p}^{n-1}\|^2+\nu\|\nabla_h\tilde{u}^{n-1}\|^2}\right]\delta_t^{\beta,*}\sqrt{\|\tilde{p}^n\|^2+\nu\|\nabla_h\tilde{u}^n\|^2}.
	\end{aligned}
\end{equation}
The Cauchy-Schwarz inequality implies that
\begin{equation}
	\begin{aligned}
		the\ right &\le \|F(u^{n,*})-F(U^{n,*})\|\|\tilde p^{n,*}\|+\|(R_1)^n\|\|\tilde p^{n,*}\|+\nu\|\nabla_h(R_2)^n\|\|\nabla_h\tilde{u}^{n,*}\|\\
		&\le \|F(u^{n,*})-F(U^{n,*})\|\|\tilde p^{n,*}\|+\sqrt{\|(R_1)^n\|^2+\nu\|\nabla_h(R_2)^n\|^2}\sqrt{\|\tilde p^{n,*}\|^2+\nu\|\nabla_h\tilde{u}^{n,*}\|^2}\\
		&\le K_1(\|\nabla_h \tilde u^{n}\|+\|\nabla_h\tilde u^{n-1}\|)\|\tilde p^{n,*}\|+\sqrt{\|(R_1)^n\|^2+\nu\|\nabla_h(R_2)^n\|^2}\sqrt{\|\tilde p^{n,*}\|^2+\nu\|\nabla_h\tilde{u}^{n,*}\|^2}.
	\end{aligned}
\end{equation}
Next, by invoking the triangle inequality for the norm, we obtain
\begin{equation}
		\begin{aligned}
		&\qquad \|\tilde p^{n,*}\|^2+\nu\|\nabla_h\tilde{u}^{n,*}\|^2 \\
&\le \left[\sigma\|\tilde p^n\|+(1-\sigma)\|\tilde p^{n-1}\|\right]^2+\nu \left[\sigma \|\nabla_h\tilde{u}^{n}\|+(1-\sigma)\|\nabla_h\tilde{u}^{n-1}\|\right]^2\\
		&=\sigma^2(\|\tilde{p}^n\|^2+\nu \|\nabla_h\tilde{u}^{n}\|^2)+(1-\sigma)^2(\|\tilde{p}^{n-1}\|^2+\nu \|\nabla_h\tilde{u}^{n-1}\|^2)\\
		&+2\sigma(1-\sigma)(\|\tilde p^n\|\|\tilde p^{n-1}\|+\nu \|\nabla_h\tilde{u}^{n}\|\|\nabla_h\tilde{u}^{n-1}\|)\\
		&=\left[\sigma\sqrt{\|\tilde{p}^n\|^2+\nu\|\nabla_h\tilde{u}^n\|^2}+(1-\sigma)\sqrt{\|\tilde{p}^{n-1}\|^2+\nu\|\nabla_h\tilde{u}^{n-1}\|^2}\right]^2\\
		&+2\sigma(1-\sigma)\left(\|\tilde p^n\|\|\tilde p^{n-1}\|+\nu \|\nabla_h\tilde{u}^{n}\|\|\nabla_h\tilde{u}^{n-1}\|-\sqrt{\|\tilde{p}^n\|^2+\nu\|\nabla_h\tilde{u}^n\|^2}\sqrt{\|\tilde{p}^{n-1}\|^2+\nu\|\nabla_h\tilde{u}^{n-1}\|^2}\right)\\
		&\le \left[\sigma\sqrt{\|\tilde{p}^n\|^2+\nu\|\nabla_h\tilde{u}^n\|^2}+(1-\sigma)\sqrt{\|\tilde{p}^{n-1}\|^2+\nu\|\nabla_h\tilde{u}^{n-1}\|^2}\right]^2.
	\end{aligned}
\end{equation}
Applying the above inequality, it yields
\begin{equation}
	\begin{aligned}
		\label{equation9}
		&\qquad (\|\nabla_h \tilde u^{n}\|+\|\nabla_h\tilde u^{n-1}\|)\|\tilde p^{n,*}\| \\
&\le \frac{1}{2}\left[\|\tilde p^{n,*}\|^2+(\|\nabla_h \tilde u^{n}\|+\|\nabla_h\tilde u^{n-1}\|)^2\right]\\
		&\le \frac{1}{2}\left\{\left[\sigma\|\tilde p^{n}\|+(1-\sigma)\|\tilde p^{n-1}\|\right]^2+\frac{1}{\nu}\max\{\frac{1}{\sigma^2},\frac{1}{(1-\sigma)^2}\}\nu\left[\sigma\|\nabla_h \tilde u^{n}\|+(1-\sigma)\|\nabla_h\tilde u^{n-1}\|\right]^2\right\}\\
		&\le \frac{1}{2}\max\left\{1,\frac{1}{\nu(1-\sigma)^2}\right\}\left\{\left[\sigma\|\tilde p^{n}\|+(1-\sigma)\|\tilde p^{n-1}\|\right]^2+\nu\left[\sigma\|\nabla_h \tilde u^{n}\|+(1-\sigma)\|\nabla_h\tilde u^{n-1}\|\right]^2\right\}\\
		&\le \frac{1}{2}\max\left\{1,\frac{1}{\nu(1-\sigma)^2}\right\}\left[\sigma\sqrt{\|\tilde{p}^n\|^2+\nu\|\nabla_h\tilde{u}^n\|^2}+(1-\sigma)\sqrt{\|\tilde{p}^{n-1}\|^2+\nu\|\nabla_h\tilde{u}^{n-1}\|^2}\right]^2.
	\end{aligned}
\end{equation}
Therefore, for $1\le n\le N$, \eqref{equation8}-\eqref{equation9} lead to
\begin{align*}
	\delta_t^{\beta,*}\sqrt{\|\tilde{p}^n\|^2+\nu\|\nabla_h\tilde{u}^n\|^2}&\le K_2\left[\sigma\sqrt{\|\tilde{p}^n\|^2+\nu\|\nabla_h\tilde{u}^n\|^2}+(1-\sigma)\sqrt{\|\tilde{p}^{n-1}\|^2+\nu\|\nabla_h\tilde{u}^{n-1}\|^2}\right]\\
	&+\sqrt{\|(R_1)^n\|^2+\nu\|\nabla_h(R_2)^n\|^2},
\end{align*}
	where	$K_2:=\frac{K_1}{2}\max\left\{1,\frac{1}{\nu(1-\sigma)^2}\right\}$.
	Applying Lemmas \ref{r1} and \ref{r3}, one has
\begin{align*}
	\sqrt{\|(R_1)^n\|^2+\nu\|\nabla_h(R_2)^n\|^2}&\le \|(R_1)^n\|+\nu^{\frac{1}{2}}\|\nabla_h(R_2)^n\|\\
	&\le C\left[h^2+(\tau_1/t_n)^{\min\{\beta+1,\frac{3-\beta}{r}\}}+(\tau_1/t_n)^{2/r}t_n^{2\beta}+\tau_1^{\beta}(\tau_1/t_n)^{\min\{\beta+1,\frac{3-\beta}{r}-\beta\}}+(\tau_1/t_n)^{2/r}t_n^{\beta}\right]\\
	&\le C \left[h^2+(\tau_1/t_n)^{\min\{\beta+1,\frac{3-\beta}{r}\}}+(\tau_1/t_n)^{2/r}t_n^{\beta}\right].
\end{align*}		
		Set $\chi_1^0=\chi_2^0=\chi_3^0$, $\delta_t^{\beta,*}\chi_1^n-\sigma K_2 \chi_1^n-(1-\sigma)K_2 \chi_1^{n-1}=h^2$,
		\begin{gather*}
			\delta_t^{\beta,*}\chi_2^n-\sigma K_2 \chi_2^n-(1-\sigma)K_2 \chi_2^{n-1}=(\tau_1/t_n)^{\min\{\beta+1,\frac{3-\beta}{r}\}},\\
			\delta_t^{\beta,*}\chi_3^n-\sigma K_2 \chi_3^n-(1-\sigma)K_2 \chi_3^{n-1}=\begin{cases}
				(\tau_1/t_n)^{2/r}t_n^{\beta},\ 2/r\ge \beta+1,\\
			\tau_1^{\beta}(\tau_1/t_n)^{2/r-\beta},\ 2/r< \beta+1.\\
			\end{cases}
		\end{gather*}		
Employing Lemma \ref{lemma6}, for $1\le n\le N$ and $r\ge 1$, it holds that
\begin{gather*}
	\chi_1^n\lesssim h^2V(n,-1),\
	\chi_2^n\lesssim
		V(n,\varphi),\
	\chi_3^n\lesssim \begin{cases}
		\tau_1^{\beta}V(n,2/r-\beta-1),\ 2/r<\beta+1,\\
		V(n,2/r-1),\ 2/r\ge \beta+1,
	\end{cases}
\end{gather*}
 where $\varphi=\min\{\beta,\frac{3-\beta}{r}-1\}$. By virtue of the comparison principle, it follows that
		\begin{align*}
			\sqrt{\|\tilde{p}^n\|^2+\nu\|\nabla_h\tilde{u}^n\|^2}\lesssim\begin{cases}
	V(n,\varphi)+\tau_1^{\beta}V(n,2/r-\beta-1)+h^2V(n,-1),\ \frac{2}{\beta+1}<r,\\
	V(n,\varphi)+V(n,2/r-1)+h^2V(n,-1),\ r\le \frac{2}{\beta+1}.		
			\end{cases}
		\end{align*}
\end{proof}
\end{theorem}
\begin{remark}
	This paper has focused solely on the time fractional sine-Gordon equation. For a general function $f(u)\in C^2(\mathbb{R})$, one can imitate the proof in Reference \cite{submittied} to further derive a convergence result on the basis of Theorem \ref{theorem1}.
\end{remark}

\begin{corollary}
	\label{corollary1}
	(Local convergence) When $t_n=T$, the local error estimates in Theorem \ref{theorem1} are as follows
	\begin{align*}
		\|\tilde{p}^n\|+\|\nabla_h \tilde{u}^n\|\lesssim
 \tau_1^{\min\{1,\frac{2}{r}\}}+h^2,\ if\  r\neq3-\beta.
	\end{align*}	
\end{corollary}

\begin{theorem}
	\label{theorem2}
Provided that the conditions in Theorem \ref{theorem1} are satisfied, and $\rho\le 7/4$, $1\le r< 4/\beta$, one has
\begin{align*}
		\|U^n\|_{\infty}\le \max_{0\le j\le N}\|u^j\|_{\infty}+1,\ for\  n=0,1,\cdots,N,
\end{align*}
when $\tau_1$ and $h$ are sufficiently small.
\begin{proof}
	Applying the fact that $\ln(x)\le \frac{1}{1-\beta}x^{1-\beta}$ for $x\ge 1$, then $\ln(t_n/\tau_1)\le \frac{1}{1-\beta}(t_n/\tau_1)^{1-\beta}$.
	From the error estimate given in Theorem \ref{theorem1}, we can deduce that
	\begin{align*}
		\|\tilde{p}^n\|+\|\nabla_h \tilde{u}^n\|\le C(\tau_1^{\min\{\beta,2/r\}}+h^2),\ for\ n\ge 0,
	\end{align*}
	where $C$ is independent of $n$ and step sizes.
	
	From \eqref{equation6a}, it holds that, for $n\ge 1$
	\begin{align*}
	\nu	\|\Delta_h\tilde{u}^{n,*}\|&\le\|R_1^n\|+\|\delta_t^{\beta,*}\tilde{p}^n\|+\|F(u^{n,*})-F(U^{n,*})\|\\
		&\le C\left(\tau_n^{-\beta}\max_{1\le j\le N}\|\tilde{p}^j\|+\max_{1\le j\le N}\|\tilde{u}^j\|+1\right)\\
		&\le  C\left(\tau_1^{-\beta}\max_{1\le j\le N}\|\tilde{p}^j\|+\max_{1\le j\le N}\|\tilde{u}^j\|+1\right),
	\end{align*}
	and
	\begin{align*}
		\|\Delta_h \tilde{u}^{n}\|\le  C\left(\tau_1^{-\beta}\max_{1\le j\le N}\|\tilde{p}^j\|+\max_{1\le j\le N}\|\tilde{u}^j\|+1\right)\le C\left[(\tau_1^{-\beta}+1)(\tau_1^{\min\{\beta,2/r\}}+h^2)+1\right].
	\end{align*}
	The derivations of above inequalities parallel that of Theorem 5.1 in \cite{submittied}; we refer the reader to that paper for details.
	Using Lemma 3.2 from \cite{MR4533497}, we get
	\begin{align*}
		\|\tilde u^n\|_{\infty}&\le C\|\tilde u^n\|^{\frac{1}{2}}(\|\tilde u^n\|+\|\Delta_h\tilde u^n\|)^{\frac{1}{2}}\\
		&\le C\left(\tau_1^{\min\{\beta,2/r\}}+h^2\right)^{\frac{1}{2}}\left[\left(\tau_1^{\min\{\beta,2/r\}}+h^2\right)+(\tau_1^{-\beta}+1)(\tau_1^{\min\{\beta,2/r\}}+h^2)+1\right]^{\frac{1}{2}}\\
		&\le C\left(\tau_1^{\min\{\beta,2/r\}}+h^2\right)^{\frac{1}{2}}\left[\tau_1^{-\beta}(\tau_1^{\min\{\beta,2/r\}}+h^2)+1\right]^{\frac{1}{2}}.
	\end{align*}
	If $\tau_1^{\min\{\beta,2/r\}}\ge h^2$, it is obviously that $\|\tilde u^n\|_{\infty}\le C\left(\tau_1^{\min\{\beta,4/r-\beta\}}+\tau_1^{\min\{\beta,2/r\}}\right)^{\frac{1}{2}}$. If $\tau_1^{\min\{\beta,2/r\}}< h^2$, Upon invoking Lemma 3.3 from \cite{MR4533497}, it follows that $\|\tilde u^n\|_{\infty}\le C h^{-1}\|\tilde u^n\|\le Ch$. Therefore, provided that $\tau_1$ and $h$ are sufficiently small, we have $\|\tilde u^n\|_{\infty}\le 1$. This completes the proof.
	
\end{proof}
\end{theorem}

\begin{lemma}(Lemma 3.2, \cite{MR4483525})
	\label{lemma7}
Let  $\sigma=1-\frac{\beta}{2}$ and properties P1, P3 and P4 hold. For arbitrary $\{v_1^n\}_{n=0}^N\ge 0$ and $\{v_2^n\}_{n=0}^N\ge 0$, we suppose that
	 \begin{align*}
		\delta_t^{\beta,*}\left[(v_1^n)^2+(v_2^n)^2\right]\le \lambda(v_1^{n,*}+v_2^{n,*})^2,\ for\ n=1,2\cdots, N,
	\end{align*}
where $\lambda>0$ and  that $\tau\le \left(4m_c\Gamma(2-\beta)\lambda\right)^{-\frac{1}{\beta}}$. Then, we have
	\begin{align*}
		v_1^n+v_2^n\le 4E_{\beta}(4\max\{1,\rho\}m_c\lambda t_n^{\beta})(v_1^0+v_2^0).
	\end{align*}
	Here, $E_{\beta}(z)=\sum_{j=0}^{\infty}\frac{z^j}{\Gamma(1+j\beta)}$.
\end{lemma}
Let $(P_1^n,U_1^n)$ denote the solution of the fully discrete scheme \eqref{equation5} with  initial data $(\psi-\hat u_0,-\hat p_0)$. Taking the difference yields the following error equation:
\begin{subequations}
	\label{equation27}
	\begin{align}
		\label{equation27a}
		& \delta_{t}^{\beta,*} \hat p_{i,j}^{n}-\nu\Delta_h \hat u_{i,j}^{n,*}+F(U_{i,j}^{n,*})-F(\hat U_{i,j}^{n,*})=0,\ for\ (x_i,y_j)\in \Omega_h,1\le n\le N,\\
		\label{equation27b}
		&\delta_{t}^{\beta,*}\hat u_{i,j}^n=\hat p_{i,j}^{n,*},\ for\ (x_i,y_j)\in \Omega_h,1\le n\le N,\\
		&\hat u_{i,j}^0=\hat u_0(x_i,y_j),\ \hat p_{i,j}^0=\hat p_0(x_i,y_j),\ for\ (x_i,y_j)\in \Omega_h,\\
		&\hat u_{i,j}^n=0,\ \hat p_{i,j}^n=0,\ for\ (x_i,y_j)\in\partial\Omega_h, 0\le n\le N,
	\end{align}
\end{subequations}
where $\hat p^n=P^n-P_1^n$, $\hat{u}^n=\bar U^n-U_1^n=U^n-\hat U^n$, and $\hat U_{i,j}^n=(U_1)_{i,j}^n+t_n\phi_{i,j}$.

\begin{theorem}(Stability analysis)
Set $\sigma=1-\frac{\beta}{2}$, $\rho\le 7/4$ and $1\le r< 4/\beta$. Assuming that Properties P1–P4 hold and the maximum time-step size $\tau$ satisfies the requirement stated in Theorem \ref{theorem1}, then, if
$\tau\le \left(8m_c\Gamma(2-\beta)K_4\right)^{-\frac{1}{\beta}}$ and $\tau_1$, $h$ are small,
 we obtain
\begin{align*}
\|\hat p^n\|+\|\nu^{\frac{1}{2}}\nabla_h\hat u^n\|\le C(\|\hat p^0\|+\|\nu^{\frac{1}{2}}\nabla_h\hat u^0\|),\ for\ n\ge 0.
\end{align*}
\end{theorem}
\begin{proof}
	Applying Theorem \ref{theorem2} and mimicking the derivation of inequality \eqref{equation28}, we obtain
	\begin{align*}
		\|F(\hat U^{n,*})-F(U^{n,*})\|\le K_3(\|\nabla_h \hat u^{n}\|+\|\nabla_h\hat u^{n-1}\|).
	\end{align*}
Here, $K_3$ depends on $\sigma$, $\kappa$ and $\max_{0\le j\le N}\|u^j\|_{\infty}$.	Taking the inner product of equation \eqref{equation27a} with $\hat{p}^{n,*}$ and equation \eqref{equation27b} with $-\nu\Delta_h\hat{u}^{n,*}$ , then summing up, we obtain
	\begin{equation}
		\begin{aligned}
			(\delta_{t}^{\beta,*} \hat p^{n},\hat p^{n,*})+\nu(\delta_{t}^{\beta,*}\nabla_h\hat u^n,\nabla_h\hat{u}^{n,*})&=(F(\hat U^{n,*})-F(U^{n,*}),\hat p^{n,*})\\
			&\le K_3 \|\hat p^{n,*}\|(\|\nabla_h \hat u^{n}\|+\|\nabla_h\hat u^{n-1}\|)\\
			&\le  K_4\left[\sigma\sqrt{\|\hat{p}^n\|^2+\nu\|\nabla_h\hat{u}^n\|^2}+(1-\sigma)\sqrt{\|\hat{p}^{n-1}\|^2+\nu\|\nabla_h\hat{u}^{n-1}\|^2}\right]^2\\
			&\le K_4\left[\sigma\|\hat{p}^n\|+(1-\sigma)\|\hat{p}^{n-1}\|+\sigma\|\nu^{\frac{1}{2}}\nabla_h\hat{u}^n\|+(1-\sigma)\|\nu^{\frac{1}{2}}\nabla_h\hat{u}^{n-1}\| \right]^2,
		\end{aligned}		
	\end{equation}
where	$K_4=\frac{K_3}{2}\max\left\{1,\frac{1}{\nu(1-\sigma)^2}\right\}$. The second inequality follows from \eqref{equation9}. Invoking the inequality stated in Lemma A.1 of \cite{MR3904430}, we arrive at, for $n\ge 1$,
\begin{align*}
\delta_{t}^{\beta,*}(\|\hat p^n\|^2+\|\nu^{\frac{1}{2}}\nabla_h\hat u^n\|^2)	\le 2K_4\left[\sigma\|\hat{p}^n\|+(1-\sigma)\|\hat{p}^{n-1}\|+\sigma\|\nu^{\frac{1}{2}}\nabla_h\hat{u}^n\|+(1-\sigma)\|\nu^{\frac{1}{2}}\nabla_h\hat{u}^{n-1}\| \right]^2.
\end{align*}
According to Lemma \ref{lemma7}, it holds that, for $n\ge 1$,
\begin{align*}
\|\hat p^n\|+\|\nu^{\frac{1}{2}}\nabla_h\hat u^n\|\le 	4E_{\beta}(8K_4\max\{1,\rho\}m_c t_n^{\beta})(\|\hat p^0\|+\|\nu^{\frac{1}{2}}\nabla_h\hat u^0\|).
\end{align*}
\end{proof}
\section{Numerical examples}\label{sec:num}
In this section we verify the local convergence orders by two numerical tests.
It has been shown in Lemmas 4  and 6 from \cite{MR3936261} that the  properties P1, P2 and P4 hold on the standard graded mesh $t_n=T(\frac{n}{N})^{r}$ when $\sigma=1-\frac{\beta}{2}$. In addition, P3 holds on this temporal mesh (Theorem 2.2, \cite{MR4270344}). We therefore employ this mesh for the time discretization in all tests.
To reduce computational errors, the integrals defining the coefficients $a_{n,k}$ and $b_{n,k}$ are calculated by Gauss-Kronrod quadrature.
 \begin{example}
 	\label{example1}
 We choose $T=0.5$, $\Omega=(0,1)^2$. The following problem is considered,
 	\begin{subequations}
 	\begin{align}
 		& D_t^{\alpha}u(\bm{x},t)-0.1\Delta u+\sin(u)=g(\bm{x},t),\ (\bm{x},t)\in \Omega\times \left(0,T\right],\\
 		&u(\bm{x},0)=\psi(\bm{x}),\ u_t(\bm{x},0)=\phi(\bm{x}),\ \bm{x}\in \Omega,\\
 		&u(\bm{x},t)=0,\ (\bm{x},t)\in \partial\Omega\times \left[0,T\right],
 	\end{align}
 \end{subequations}
 where $1<\alpha<2$, the forcing term $g(\bm{x},t)$, $\psi(\bm{x})$ and $\phi(\bm{x})$ are up to the exact solution
 $u=(t+t^{\alpha})\sin(\pi x)\sin(\pi y)$.
 \end{example}
 We define $E_L^N=\|\nabla_h \tilde{u}^N\|$, and
 the convergence rate is computed by
 \begin{align*}
 	rate=\log_2{\frac{E_L^N}{E_L^{2N}}}.
 \end{align*}
First, we fix $N=800$ to render the temporal error negligible, so that the spatial error becomes dominant. Table \ref{table1} represents the errors $E_L^N$ with $\alpha=1.5$ and $r=1.5$, and the observed spatial convergence rates coincide with order predicted by Corollary \ref{corollary1}.  Since the spatial error is $O(h^2)$, we set $M=N$. The local temporal convergence orders in Corollary \ref{corollary1} is well reflected by Tables \ref{table2}-\ref{table4}. The numerical experiments also demonstrate that the fully discrete scheme achieves the optimal local convergence order once $r\ge 2$.
\begin{table}[H]
	\centering
	\caption{Errors and convergence rates for Example \ref{example1} (dominated by spatial error).}
		\label{table1}
	\begin{tabular}{ccccccc}
		\toprule
	& $M$ & 4 & 8 & 16 & 32&expected order \\
		\midrule

		&$E_L^N$ & 5.3313e-03 &1.3727e-03 &3.4252e-04 &8.2408e-05   \\
		&rate  &  &  1.9575    &2.0027    &2.0553  &2\\
		\bottomrule
	\end{tabular}
\end{table}

 \begin{table}[H]
 	\centering
 	\caption{Local errors and convergence rates with $r=1$ for Example \ref{example1}.}
 	\label{table2}
 	\begin{tabular}{ccccccc}
 		\toprule
 		$M=N$	& $N$ & 16 & 32 & 64 & 128&expected order \\
 		\midrule
 		\multirow{6}{*}{$E_L^N$}
 		& $\alpha = 1.1$ & 3.1876e-03 & 1.7222e-03 & 8.9359e-04 & 4.5512e-04  \\
 		&  &  & 0.8882 & 0.9466 & 0.9733  &1\\
 		& $\alpha = 1.5$ & 4.3486e-03 & 2.1507e-03 & 1.0670e-03 & 5.3091e-04  \\
 		&  &  & 1.0157 & 1.0112 & 1.0071  &1\\
 		& $\alpha = 1.9$  &1.1931e-03 &5.6317e-04 &2.7287e-04 &1.3419e-04 \\
 		&&&1.0830    &1.0454   & 1.0240&1\\
 		\bottomrule
 	\end{tabular}
 \end{table}

 \begin{table}[H]
	\centering
	\caption{Local errors and convergence rates with $r=1.5$ for Example \ref{example1}.}
	\begin{tabular}{ccccccc}
		\toprule
		$M=N$	& $N$ & 16 & 32 & 64 & 128&expected order \\
		\midrule
		\multirow{6}{*}{$E_L^N$}
		& $\alpha = 1.1$ & 8.1494e-04 &3.2829e-04 &1.2611e-04 &4.7144e-05  \\
		&  &  & 1.3117    &1.3803    &1.4195
		  &1.5\\
		& $\alpha = 1.5$ & 1.6334e-03 &5.7432e-04 &2.0081e-04 &7.0167e-05   \\
		&  &  & 1.5080    &1.5160    &1.5170  &1.5\\
		& $\alpha = 1.9$  &5.3904e-04 &1.8026e-04 &6.0908e-05 &2.0790e-05  \\
		&&& 1.5803    &1.5654    &1.5508&1.5\\
		\bottomrule
	\end{tabular}
\end{table}

\begin{table}[H]
	\centering
	\caption{Local errors and convergence rates with $r=2$ for Example \ref{example1}.}
	\label{table4}
	\begin{tabular}{ccccccc}
		\toprule
		$M=N$	& $N$ & 16 & 32 & 64 & 128&expected order \\
		\midrule
		\multirow{6}{*}{$E_L^N$}
		& $\alpha = 1.1$ & 5.1921e-04 &1.4347e-04 &3.7963e-05 &9.8151e-06   \\
		&  &  &  1.8556    &1.9181    &1.9515
		&2\\
		& $\alpha = 1.5$ & 1.1444e-03 &3.1736e-04 &8.5638e-05 &2.2697e-05   \\
		&  &  &   1.8504    &1.8898   & 1.9158
		 &2\\
		& $\alpha = 1.9$  &4.6356e-04 &1.2709e-04 &3.4371e-05 &9.2109e-06  \\
		&&& 1.8669    &1.8865   & 1.8998
		&2\\
		\bottomrule
	\end{tabular}
\end{table}

 \begin{example}
	\label{example2}
	We consider the problem \eqref{equation1} with $T=0.5$, $L=\pi$, $\kappa=1$ and $\nu=1$, whose exact solution is unknown. The initial data is set to $\psi(\bm{x})=\exp(\cos(x)\cos(y))$ and $\phi(\bm{x})=0$.
\end{example}
Since the exact solution is unknown, the error is calculated by the two-mesh method\cite{MR1750671}; specifically, we define
\begin{align*}
	E_L^N=\|\nabla_h (U^N-V^{2N})\|,
\end{align*}
where $V^n$ is the numerical solution on the mesh $t_n=T(\frac{n}{2N})^r$.
Tables \ref{6.5}-\ref{6.7} show the local errors and convergence rates for Example \ref{example2} using different $r$ and $\alpha$, which are consistent with Corollary \ref{corollary1} as well.
\begin{table}[H]
	\centering
	\caption{Local errors and convergence rates with $r=1$  for Example \ref{example2}.}
	\label{6.5}
	\begin{tabular}{cccccccc}
		\toprule
		$M=25$	& $N$ & 32 & 64 & 128 & 256&512&expected order \\
		\midrule
		\multirow{6}{*}{$E_L^N$}
		& $\alpha = 1.1$ & 1.0657e-02 &5.2836e-03 &2.6337e-03 &1.3154e-03 &6.5741e-04     \\
		&  &  &    1.0122   & 1.0044   & 1.0016   & 1.0006
		&1\\
		& $\alpha = 1.5$ & 4.2853e-02 &1.9864e-02 &9.5653e-03 &4.6928e-03 &2.3240e-03    \\
		&  &  &   1.1093    &1.0542   & 1.0274    &1.0139
		&1\\
		& $\alpha = 1.9$  &3.3683e-01 &1.4621e-01 &6.7578e-02 &3.2411e-02& 1.5860e-02    \\
		&&&  1.2040    &1.1134   & 1.0601    &1.0311
		&1\\
		\bottomrule
	\end{tabular}
\end{table}

\begin{table}[H]
	\centering
	\caption{Local errors and convergence rates with $r=1.5$ for Example \ref{example2}.}
	\begin{tabular}{cccccccc}
		\toprule
		$M=25$	& $N$ & 32 & 64 & 128 & 256&512 &expected order \\
		\midrule
		\multirow{6}{*}{$E_L^N$}
		& $\alpha = 1.1$ & 3.0512e-03 &1.0309e-03 &3.5646e-04 &1.2489e-04 &4.4052e-05     \\
		&  &  &      1.5654   & 1.5321    &1.5131    &1.5033
		&1.5\\
		& $\alpha = 1.5$ & 1.8113e-02 &5.3896e-03 &1.6728e-03 &5.3588e-04 &1.7621e-04      \\
		&  &  &   1.7488    &1.6879    &1.6423    &1.6047
		&1.5\\
		& $\alpha = 1.9$  &1.8030e-01 &5.1088e-02 &1.4906e-02 &4.4824e-03 &1.3882e-03    \\
		&&&
		1.8193    &1.7771    &1.7335   & 1.6911
		&1.5\\
		\bottomrule
	\end{tabular}
\end{table}

\begin{table}[H]
	\centering
	\caption{Local errors and convergence rates with $r=2$  for Example \ref{example2}.}
	\label{6.7}
	\begin{tabular}{cccccccc}
		\toprule
		$M=25$	& $N$ & 32 & 64 & 128 & 256&512 &expected order \\
		\midrule
		\multirow{6}{*}{$E_L^N$}
		& $\alpha = 1.1$ & 2.2028e-03 &5.6130e-04 &1.4248e-04 &3.5994e-05 &9.0601e-06      \\
		&  &  &   1.9725   & 1.9780    &1.9849   & 1.9902
		&2\\
		& $\alpha = 1.5$ & 1.8519e-02 &4.6949e-03 &1.2080e-03 &.1081e-04 &7.9682e-05     \\
		&  &  &   1.9798    &1.9584   & 1.9586    &1.9637
		
		&2\\
		& $\alpha = 1.9$  &1.9796e-01 &5.0357e-02 &1.2767e-02 &3.2349e-03 &8.1944e-04     \\
		&&&
		1.9749   & 1.9798    &1.9806   & 1.9810
		&2\\
		\bottomrule
	\end{tabular}
\end{table}

\bibliographystyle{unsrt}
\bibliography{ref}

 \end{document}